\documentclass[reqno,onecolumn]{amsart}
\usepackage{amsfonts,amscd,amsthm}
\usepackage{amssymb}
\usepackage{amsmath}
\usepackage{graphicx}
\usepackage{amscd,amsthm}

\newtheorem{theorem}{Theorem}
\theoremstyle{plain}

\newtheorem{corollary}{Corollary}

\newtheorem{definition}{Definition}

\newtheorem{lemma}{Lemma}

\newtheorem{proposition}{Proposition}
\newtheorem{remark}{Remark}

\numberwithin{equation}{section}

\input{tcilatex}

\begin{document}

\begin{center}
\pagestyle{myheadings}\thispagestyle{empty}%
\markboth{\bf  N. Saba and A. Boussayoud }
{\bf SOME NEW THEOREMS ON GENERATING FUNCTIONS AND THEIR APPLICATIONS ON ODD AND...}

\textbf{SOME NEW\ THEOREMS ON GENERATING FUNCTIONS AND THEIR APPLICATIONS ON
ODD AND EVEN\emph{\ }CERTAIN NUMBERS ATTACHED TO }$p$\textbf{\ AND }$q$%
\textbf{\ PARAMETERS}

\textbf{Nabiha Saba}

LMAM Laboratory and Department of Mathematics,

Mohamed Seddik Ben Yahia University, Jijel, Algeria

\textbf{E-Mail: sabarnhf1994@gmail.com}

\textbf{Ali Boussayoud}$^{\ast }$

LMAM Laboratory and Department of Mathematics,

Mohamed Seddik Ben Yahia University, Jijel, Algeria

\textbf{E-Mail: aboussayoud@yahoo.fr}

$^{\ast }$\textbf{Corresponding author}\ 

\textbf{\large Abstract}
\end{center}

\begin{quotation}
\qquad In this study, we first provide some\ new theorems by using the
symmetrizing operator $\delta _{e_{1}e_{2}}^{k}$\ for $k\in \left\{
0,1,2,3,4\right\} $. After that, by using this theorems we introduce a new
family of generating functions of odd and even terms of $\left( p,q\right) $%
-Fibonacci numbers, $\left( p,q\right) $-Lucas numbers, $\left( p,q\right) $%
-Pell numbers,\ $\left( p,q\right) $-Pell Lucas numbers, $\left( p,q\right) $%
-Jacobsthal numbers and $\left( p,q\right) $-Jacobsthal Lucas numbers. Then,
we give the new generating functions of the products of these $\left(
p,q\right) $-numbers with odd and even phrases\emph{\ }of $\left( p,q\right) 
$-numbers.
\end{quotation}



\section{\textbf{Introduction and preliminary results}}

The generating function can be applied to solve many problems in
mathematics. This concept makes it possible to know the terms of a recurring
and linear sequence with constant coefficients. This procedure demonstrates
how one can find the terms of a recurrent and linear sequence without
calculating the previous terms. In this work, we investigate the generating
functions of odd and even $\left( p,q\right) $-numbers and some products of
them. In this section, we first present the $(p,q)$-numbers and some
generalization of them.

For over several years, there are many recursive sequences that have been
studied in the literatures. the famous examples of these sequences are $%
\left( p,q\right) $-Fibonacci and $\left( p,q\right) $-Lucas numbers, $%
\left( p,q\right) $-Jacobsthal\ and $\left( p,q\right) $-Jacobsthal Lucas
numbers, $\left( p,q\right) $-Pell and $\left( p,q\right) $-Pell Lucas
numbers (see \cite{SUVAR, SU, S, HASAN, UYG, NSABA, NABIHA}), because they
are extensively used in various research areas. The authors in \cite{NS}
defined the generalized $\left( p,q\right) $-Fibonacci sequence $\left\{
f_{p,q,n}\left( \alpha ,\beta ,\gamma \right) \right\} _{n\geq 0}$,
generalized $\left( p,q\right) $-Pell sequence $\left\{ l_{p,q,n}\left(
\alpha ,\beta ,\gamma \right) \right\} _{n\geq 0}$\ and generalized $\left(
p,q\right) $-Jacobsthal sequence $\left\{ C_{p,q,n}\left( \alpha ,\beta
,\gamma \right) \right\} _{n\geq 0}$ as follows:%
\begin{equation}
f_{p,q,n}:=\left\{ 
\begin{array}{c}
\alpha ,\text{ \ \ \ \ \ \ \ \ \ \ \ \ \ \ \ \ \ \ \ \ \ \ \ \ \ \ \ \ \ \ \
\ \ \ \ \ \ if }n=0 \\ 
\beta +\gamma p,\text{ \ \ \ \ \ \ \ \ \ \ \ \ \ \ \ \ \ \ \ \ \ \ \ \ \ \ \
\ \ \ if }n=1 \\ 
pf_{p,q,n-1}+qf_{p,q,n-2},\text{\ \ \ \ \ \ \ \ \ \ \ \ \ if }n\geq 2%
\end{array}%
\right. ,  \tag{1.1}
\end{equation}%
\begin{equation}
l_{p,q,n}:=\left\{ 
\begin{array}{c}
\alpha ,\text{ \ \ \ \ \ \ \ \ \ \ \ \ \ \ \ \ \ \ \ \ \ \ \ \ \ \ \ \ \ \ \
\ \ \ \ \ if }n=0 \\ 
\beta +2\gamma p,\text{ \ \ \ \ \ \ \ \ \ \ \ \ \ \ \ \ \ \ \ \ \ \ \ \ \ \
\ if }n=1 \\ 
2pl_{p,q,n-1}+ql_{p,q,n-2},\text{\ \ \ \ \ \ \ \ \ \ \ if }n\geq 2%
\end{array}%
\right. ,  \tag{1.2}
\end{equation}%
and%
\begin{equation}
C_{p,q,n}:=\left\{ 
\begin{array}{c}
\alpha ,\text{ \ \ \ \ \ \ \ \ \ \ \ \ \ \ \ \ \ \ \ \ \ \ \ \ \ \ \ \ \ \ \
\ \ if }n=0 \\ 
\beta +\gamma p,\text{ \ \ \ \ \ \ \ \ \ \ \ \ \ \ \ \ \ \ \ \ \ \ \ \ \ \
if }n=1 \\ 
pC_{p,q,n-1}+2qC_{p,q,n-2},\text{\ \ \ \ \ \ if }n\geq 2%
\end{array}%
\right. .  \tag{1.3}
\end{equation}%
The particular cases of Eqs. (1.1), (1.2) and (1.3) are listed below:

\begin{remark}
Let $\alpha =\gamma =0$ and $\beta =1$ in Eqs. (1.1), (1.2) and (1.3), then
we get:
\end{remark}

- $\left( p,q\right) $-\textit{Fibonacci numbers} $\left\{ F_{p,q,n}\right\}
_{n\geq 0}$, defined recursively by:%
\begin{equation}
\left\{ 
\begin{array}{l}
F_{p,q,0}=0,\text{ }F_{p,q,1}=1 \\ 
F_{p,q,n}=pF_{p,q,n-1}+qF_{p,q,n-2}\text{ }\left( n\geq 2\right)%
\end{array}%
\right. .  \tag{1.4}
\end{equation}

- $\left( p,q\right) $-\textit{Pell numbers} $\left\{ P_{p,q,n}\right\}
_{n\geq 0}$, defined recursively by:%
\begin{equation}
\left\{ 
\begin{array}{l}
P_{p,q,0}=0,\text{ }P_{p,q,1}=1 \\ 
P_{p,q,n}=2pP_{p,q,n-1}+qP_{p,q,n-2}\text{ }\left( n\geq 2\right)%
\end{array}%
\right. .  \tag{1.5}
\end{equation}

- $\left( p,q\right) $-\textit{Jacobsthal numbers} $\left\{
J_{p,q,n}\right\} _{n\geq 0}$, defined recursively by:%
\begin{equation}
\left\{ 
\begin{array}{l}
J_{p,q,0}=0,\text{ }J_{p,q,1}=1 \\ 
J_{p,q,n}=pJ_{p,q,n-1}+2qJ_{p,q,n-2}\text{ }\left( n\geq 2\right)%
\end{array}%
\right. .  \tag{1.6}
\end{equation}

\begin{remark}
Let $\alpha =2,$ $\gamma =1$ and $\beta =0$ in Eqs. (1.1), (1.2) and (1.3),
then we get:
\end{remark}

- $\left( p,q\right) $-\textit{Lucas} \textit{numbers} $\left\{
L_{p,q,n}\right\} _{n\geq 0}$, defined recursively by:%
\begin{equation}
\left\{ 
\begin{array}{l}
L_{p,q,0}=2,\text{ }L_{p,q,1}=p \\ 
L_{p,q,n}=pL_{p,q,n-1}+qL_{p,q,n-2}\text{ }\left( n\geq 2\right)%
\end{array}%
\right. .  \tag{1.7}
\end{equation}

- $\left( p,q\right) $-\textit{Pell Lucas numbers} $\left\{
Q_{p,q,n}\right\} _{n\geq 0}$, defined recursively by:%
\begin{equation}
\left\{ 
\begin{array}{l}
Q_{p,q,0}=2,\text{ }Q_{p,q,1}=2p \\ 
Q_{p,q,n}=2pQ_{p,q,n-1}+qQ_{p,q,n-2}\text{ }\left( n\geq 2\right)%
\end{array}%
\right. .  \tag{1.8}
\end{equation}

- $\left( p,q\right) $-\textit{Jacobsthal Lucas numbers} $\left\{
j_{p,q,n}\right\} _{n\geq 0}$, defined recursively by:%
\begin{equation}
\left\{ 
\begin{array}{l}
j_{p,q,0}=2,\text{ }j_{p,q,1}=p \\ 
j_{p,q,n}=pj_{p,q,n-1}+2qj_{p,q,n-2}\text{ }\left( n\geq 2\right)%
\end{array}%
\right. .  \tag{1.9}
\end{equation}

In\ fact,\ these $\left( p,q\right) $-numbers\ are\ the generalization of
some $k$-numbers,\ like\ $k$-Fibonacci,\ $k$-Pell,\ $k$-Jacobsthal, $k$%
-Lucas, $k$-Pell Lucas and\ $k$-Jacobsthal\ Lucas\ numbers,\ the following
names, notations and recurrence relations (see Tab. 1) are used for the
special cases for $p$ and$\ q$ of Eqs. (1.4)-(1.9) respectively.%
\begin{equation*}
\begin{tabular}{|c|c|c|c|c|}
\hline
$p$ & $q$ & $k$\textbf{-numbers} & \textbf{Notation} & \textbf{Recurrence
relation} \\ \hline
$k$ & $1$ & $k$-Fibonacci & $F_{k,n}$ & $\left\{ 
\begin{array}{l}
F_{k,0}=0,\text{ }F_{k,1}=1 \\ 
F_{k,n}=kF_{k,n-1}+F_{k,n-2}\text{ }\left( n\geq 2\right)%
\end{array}%
\right. $ \\ \hline
$1$ & $k$ & $k$-Pell & $P_{k,n}$ & $\left\{ 
\begin{array}{l}
P_{k,0}=0,\text{ }P_{k,1}=1 \\ 
P_{k,n}=2P_{k,n-1}+kP_{k,n-2}\text{ }\left( n\geq 2\right)%
\end{array}%
\right. $ \\ \hline
$k$ & $1$ & $k$-Jacobsthal & $J_{k,n}$ & $\left\{ 
\begin{array}{l}
J_{k,0}=0,\text{ }J_{k,1}=1 \\ 
J_{k,n}=kJ_{k,n-1}+2J_{k,n-2}\text{ }\left( n\geq 2\right)%
\end{array}%
\right. $ \\ \hline
$k$ & $1$ & $k$-Lucas & $L_{k,n}$ & $\left\{ 
\begin{array}{l}
L_{k,0}=2,\text{ }L_{k,1}=k \\ 
L_{k,n}=kL_{k,n-1}+L_{k,n-2}\text{ }\left( n\geq 2\right)%
\end{array}%
\right. $ \\ \hline
$1$ & $k$ & $k$-Pell Lucas & $Q_{k,n}$ & $\left\{ 
\begin{array}{l}
Q_{k,0}=2,\text{ }Q_{k,1}=2 \\ 
Q_{k,n}=2Q_{k,n-1}+kQ_{k,n-2}\text{ }\left( n\geq 2\right)%
\end{array}%
\right. $ \\ \hline
$k$ & $1$ & $k$-Jacobsthal Lucas & $j_{k,n}$ & $\left\{ 
\begin{array}{l}
j_{k,0}=2,\text{ }j_{k,1}=k \\ 
j_{k,n}=kj_{k,n-1}+2j_{k,n-2}\text{ }\left( n\geq 2\right)%
\end{array}%
\right. $ \\ \hline
\end{tabular}%
\end{equation*}

\begin{center}
\textbf{Table 1. }Some\textbf{\ }$k$-numbers.
\end{center}

For the special case\emph{\ }$k=1$ in the Tab.\ 1, the\emph{\ }sequences of
Fibonacci, Pell, Jacobsthal, Lucas, Pell Lucas and Jacobsthal Lucas\
numbers\ are\ respectively obtained.

Next, we present some backgrounds and results about the\ symmetric functions.

\begin{definition}
\cite{Merca}\ Let $k$ and $n$ be two positive integers and $\left\{
a_{1},a_{2},...,a_{n}\right\} $ are set of given variables the $k^{th}$
complete homogeneous symmetric function $h_{k}\left(
a_{1},a_{2},...,a_{n}\right) $ is defined by:%
\begin{equation*}
h_{k}\left( a_{1},a_{2},...,a_{n}\right) =\underset{i_{1}+i_{2}+...+i_{n}=k}{%
\sum }a_{1}^{i_{1}}a_{2}^{i_{2}}...a_{n}^{i_{n}}\text{ \ \ \ \ \ \ \ \ \ \ }%
\left( k\geq 0\right) ,
\end{equation*}%
with $i_{1},i_{2},...,i_{n}\geq 0.$
\end{definition}

\begin{remark}
Set $h_{0}\left( a_{1},a_{2},...,a_{n}\right) =1,$ by usual convention. For $%
k<0,$ we set $h_{k}\left( a_{1},a_{2},...,a_{n}\right) =0.$
\end{remark}

\begin{definition}
\cite{Abderrezzak}\ Let\ $A$\ and $E$ be any two alphabets. We define $%
S_{n}(A-E)$ by the following form:%
\begin{equation}
\frac{\tprod\limits_{e\in E}(1-ez)}{\tprod\limits_{a\in A}(1-az)}%
=\sum\limits_{n=0}^{\infty }S_{n}(A-E)z^{n},  \tag{1.10}
\end{equation}%
with the condition $S_{n}(A-E)=0\ $for\ $n<0.$
\end{definition}

Equation (1.10) can be rewritten in the following form:%
\begin{equation*}
\dsum\limits_{n=0}^{\infty }S_{n}(A-E)z^{n}=\left(
\dsum\limits_{n=0}^{\infty }S_{n}(A)z^{n}\right) \times \left(
\dsum\limits_{n=0}^{\infty }S_{n}(-E)z^{n}\right) ,
\end{equation*}%
where 
\begin{equation*}
S_{n}(A-E)=\dsum\limits_{j=0}^{n}S_{n-j}(-E)S_{j}(A).
\end{equation*}

\begin{definition}
\cite{Macdonald} Given a function $f$ on $\mathbb{R}^{n}$, the divided
difference operator is defined as follows:%
\begin{equation*}
\partial _{e_{i}e_{i+1}}(f)=\frac{f(e_{1},\cdots ,e_{i},e_{i+1},\cdots
,e_{n})-f(e_{1},\cdots ,e_{i-1},e_{i+1},e_{i},e_{i+2},\cdots ,e_{n})}{%
e_{i}-e_{i+1}}.
\end{equation*}
\end{definition}

\begin{definition}
\cite{Nabiha 2} Let $n$ be positive integer and\ $E=\left\{
e_{1},e_{2}\right\} \ $are set of given variables. Then, the $n^{th}$
symmetric function $S_{n}(e_{1}+e_{2})$ is defined by:%
\begin{equation*}
S_{n}(E)=S_{n}(e_{1}+e_{2})=\frac{e_{1}^{n+1}-e_{2}^{n+1}}{e_{1}-e_{2}},
\end{equation*}%
with%
\begin{eqnarray*}
S_{0}(E) &=&S_{0}(e_{1}+e_{2})=1, \\
S_{1}(E) &=&S_{1}(e_{1}+e_{2})=e_{1}+e_{2}, \\
S_{2}(E) &=&S_{2}(e_{1}+e_{2})=e_{1}^{2}+e_{1}e_{2}+e_{2}^{2}, \\
&&\vdots
\end{eqnarray*}
\end{definition}

\begin{proposition}
\cite{NABIHA} For\ $n\in 
\mathbb{N}
,$ the symmetric functions of $\left( p,q\right) $-Fibonacci and $\left(
p,q\right) $-Lucas numbers, $\left( p,q\right) $-Jacobsthal\ and $\left(
p,q\right) $-Jacobsthal Lucas numbers, $\left( p,q\right) $-Pell and $\left(
p,q\right) $-Pell Lucas numbers are given by:%
\begin{eqnarray*}
F_{p,q,n} &=&S_{n-1}\left( e_{1}+\left[ -e_{2}\right] \right) \text{ and }%
L_{p,q,n}=2S_{n}\left( e_{1}+\left[ -e_{2}\right] \right) -pS_{n-1}\left(
e_{1}+\left[ -e_{2}\right] \right) ,\text{ with }e_{1,2}=\frac{p\pm \sqrt{%
p^{2}+4q}}{2}. \\
J_{p,q,n} &=&S_{n-1}\left( e_{1}+\left[ -e_{2}\right] \right) \text{ and }%
j_{p,q,n}=2S_{n}\left( e_{1}+\left[ -e_{2}\right] \right) -pS_{n-1}\left(
e_{1}+\left[ -e_{2}\right] \right) ,\text{ with }e_{1,2}=\frac{p\pm \sqrt{%
p^{2}+8q}}{2}. \\
P_{p,q,n} &=&S_{n-1}\left( e_{1}+\left[ -e_{2}\right] \right) \text{ and }%
Q_{p,q,n}=2S_{n}\left( e_{1}+\left[ -e_{2}\right] \right) -2pS_{n-1}\left(
e_{1}+\left[ -e_{2}\right] \right) ,\text{ with }e_{1,2}=p\pm \sqrt{p^{2}+q}.
\end{eqnarray*}
\end{proposition}

\begin{definition}
\cite{Bouss5}$\ $Given an alphabet $E=\left\{ e_{1},e_{2}\right\} $, the
symmetrizing operator $\delta _{e_{1}e_{2}}^{k}$ is defined by:%
\begin{equation}
\delta _{e_{1}e_{2}}^{k}(f)=\dfrac{e_{1}^{k}f(e_{1})-e_{2}^{k}f(e_{2})}{%
e_{1}-e_{2}},\ \left( k\in 
\mathbb{N}
_{0}:=%
\mathbb{N}
\cup \left\{ 0\right\} =\left\{ 0,1,2,3,\cdots \right\} \right) .  \tag{1.11}
\end{equation}
\end{definition}

\begin{remark}
If$\ k=0,$ the operator $(1.11)$ gives us:%
\begin{eqnarray*}
\delta _{e_{1}e_{2}}^{0}(f) &=&\dfrac{f(e_{1})-f(e_{2})}{e_{1}-e_{2}} \\
&=&\partial _{e_{1}e_{2}}(f).
\end{eqnarray*}
\end{remark}

In the present article, the sections are organized as follows: Section 2
gives new theorems by using the symmetrizing operator $\delta
_{e_{1}e_{2}}^{k}.$ By using the theorems given in the previous section, the
new\ generating functions for odd and even terms of\emph{\ }$(p,q)$-numbers
are constructed in Section 3. Section 4 gives the new generating functions
for the products of $\left( p,q\right) $-numbers with odd and even $\left(
p,q\right) $-numbers.\ Finally, some concluded remarks are discussed in
Section 5.

\section{\textbf{New theorems on symmetric functions and their proofs}}

In this part, we are now in a position to provide some\emph{\ }new theorems
by using the symmetrizing operator $\delta _{e_{1}e_{2}}^{k}$ for $k\in
\left\{ 0,1,2,3,4\right\} .$ We now begin with the following theorem.

\begin{theorem}
Given two alphabets $A=\left\{ a_{1},a_{2}\right\} $ and $E=\left\{
e_{1},e_{2}\right\} ,$ we have:%
\begin{equation}
\sum\limits_{n=0}^{\infty }S_{n}\left( A\right) S_{2n-1}\left( E\right)
z^{n}=\frac{\left( a_{1}+a_{2}\right) \left( e_{1}+e_{2}\right)
z-a_{1}a_{2}\left( e_{1}+e_{2}\right) \left( \left( e_{1}+e_{2}\right)
^{2}-2e_{1}e_{2}\right) z^{2}}{\left( \dsum\limits_{n=0}^{\infty
}S_{n}\left( -A\right) e_{1}^{2n}z^{n}\right) \left(
\dsum\limits_{n=0}^{\infty }S_{n}\left( -A\right) e_{2}^{2n}z^{n}\right) }. 
\tag{2.1}
\end{equation}
\end{theorem}

\begin{proof}
By applying the divided difference operator $\partial _{e_{1}e_{2}}\ $to the
series$\ f\left( e_{1}z\right) =\dsum\limits_{n=0}^{\infty }S_{n}\left(
A\right) e_{1}^{2n}z^{n}$, the left-hand side of the formula (2.1) can be
written as:%
\begin{eqnarray*}
\partial _{e_{1}e_{2}}f\left( e_{1}z\right) &=&\frac{\dsum\limits_{n=0}^{%
\infty }S_{n}\left( A\right) e_{1}^{2n}z^{n}-\dsum\limits_{n=0}^{\infty
}S_{n}\left( A\right) e_{2}^{2n}z^{n}}{e_{1}-e_{2}} \\
&=&\sum_{n=0}^{\infty }S_{n}\left( A\right) \left( \frac{%
e_{1}^{2n}-e_{2}^{2n}}{e_{1}-e_{2}}\right) z^{n} \\
&=&\sum_{n=0}^{\infty }S_{n}\left( A\right) S_{2n-1}\left( E\right) z^{n}.
\end{eqnarray*}

By applying the divided difference operator $\partial _{e_{1}e_{2}}\ $to the
series$\ f\left( e_{1}z\right) =\frac{1}{\dsum\limits_{n=0}^{\infty
}S_{n}\left( -A\right) e_{1}^{2n}z^{n}}$, the right-hand side of the formula
(2.1) can be expressed as:%
\begin{eqnarray*}
\partial _{e_{1}e_{2}}f\left( e_{1}z\right) &=&\frac{\frac{1}{%
\dsum\limits_{n=0}^{\infty }S_{n}\left( -A\right) e_{1}^{2n}z^{n}}-\frac{1}{%
\dsum\limits_{n=0}^{\infty }S_{n}\left( -A\right) e_{2}^{2n}z^{n}}}{%
e_{1}-e_{2}} \\
&=&\frac{\dsum\limits_{n=0}^{\infty }S_{n}\left( -A\right)
e_{2}^{2n}z^{n}-\dsum\limits_{n=0}^{\infty }S_{n}\left( -A\right)
e_{1}^{2n}z^{n}}{\left( e_{1}-e_{2}\right) \left( \dsum\limits_{n=0}^{\infty
}S_{n}\left( -A\right) e_{1}^{2n}z^{n}\right) \left(
\dsum\limits_{n=0}^{\infty }S_{n}\left( -A\right) e_{2}^{2n}z^{n}\right) } \\
&=&\frac{\dsum\limits_{n=0}^{\infty }S_{n}\left( -A\right) \frac{%
e_{2}^{2n}-e_{1}^{2n}}{e_{1}-e_{2}}z^{n}}{\left( \dsum\limits_{n=0}^{\infty
}S_{n}\left( -A\right) e_{1}^{2n}z^{n}\right) \left(
\dsum\limits_{n=0}^{\infty }S_{n}\left( -A\right) e_{2}^{2n}z^{n}\right) }
\end{eqnarray*}%
\begin{eqnarray*}
&=&\frac{S_{1}\left( -A\right) \frac{e_{2}^{2}-e_{1}^{2}}{e_{1}-e_{2}}%
z+S_{2}\left( -A\right) \frac{e_{2}^{4}-e_{1}^{4}}{e_{1}-e_{2}}z^{2}}{\left(
\dsum\limits_{n=0}^{\infty }S_{n}\left( -A\right) e_{1}^{2n}z^{n}\right)
\left( \dsum\limits_{n=0}^{\infty }S_{n}\left( -A\right)
e_{2}^{2n}z^{n}\right) } \\
&=&\frac{-S_{1}\left( -A\right) \left( e_{1}+e_{2}\right) z-S_{2}\left(
-A\right) \left( \left( e_{1}+e_{2}\right) ^{3}-2e_{1}e_{2}\left(
e_{1}+e_{2}\right) \right) z^{2}}{\left( \dsum\limits_{n=0}^{\infty
}S_{n}\left( -A\right) e_{1}^{2n}z^{n}\right) \left(
\dsum\limits_{n=0}^{\infty }S_{n}\left( -A\right) e_{2}^{2n}z^{n}\right) } \\
&=&\frac{\left( a_{1}+a_{2}\right) \left( e_{1}+e_{2}\right)
z-a_{1}a_{2}\left( e_{1}+e_{2}\right) \left( \left( e_{1}+e_{2}\right)
^{2}-2e_{1}e_{2}\right) z^{2}}{\left( \dsum\limits_{n=0}^{\infty
}S_{n}\left( -A\right) e_{1}^{2n}z^{n}\right) \left(
\dsum\limits_{n=0}^{\infty }S_{n}\left( -A\right) e_{2}^{2n}z^{n}\right) }.
\end{eqnarray*}

Thus, this completes the proof.
\end{proof}

\begin{theorem}
Given two alphabets $A=\left\{ a_{1},a_{2}\right\} $ and $E=\left\{
e_{1},e_{2}\right\} ,$ we have:%
\begin{equation}
\sum\limits_{n=0}^{\infty }S_{n}\left( A\right) S_{2n}\left( E\right) z^{n}=%
\frac{1+e_{1}e_{2}\left( a_{1}+a_{2}\right) z-e_{1}e_{2}a_{1}a_{2}\left(
\left( e_{1}+e_{2}\right) ^{2}-e_{1}e_{2}\right) z^{2}}{\left(
\dsum\limits_{n=0}^{\infty }S_{n}\left( -A\right) e_{1}^{2n}z^{n}\right)
\left( \dsum\limits_{n=0}^{\infty }S_{n}\left( -A\right)
e_{2}^{2n}z^{n}\right) }.  \tag{2.2}
\end{equation}
\end{theorem}

\begin{proof}
By applying the operator $\delta _{e_{1}e_{2}}\ $to the series$\ f\left(
e_{1}z\right) =\dsum\limits_{n=0}^{\infty }S_{n}\left( A\right)
e_{1}^{2n}z^{n}$, the left-hand side of the formula (2.2) can be written as:%
\begin{eqnarray*}
\delta _{e_{1}e_{2}}f\left( e_{1}z\right) &=&\frac{e_{1}\dsum\limits_{n=0}^{%
\infty }S_{n}\left( A\right) e_{1}^{2n}z^{n}-e_{2}\dsum\limits_{n=0}^{\infty
}S_{n}\left( A\right) e_{2}^{2n}z^{n}}{e_{1}-e_{2}} \\
&=&\sum_{n=0}^{\infty }S_{n}\left( A\right) \left( \frac{%
e_{1}^{2n+1}-e_{2}^{2n+1}}{e_{1}-e_{2}}\right) z^{n} \\
&=&\sum_{n=0}^{\infty }S_{n}\left( A\right) S_{2n}\left( E\right) z^{n}.
\end{eqnarray*}

By applying the operator $\delta _{e_{1}e_{2}}\ $to the series$\ f\left(
e_{1}z\right) =\frac{1}{\dsum\limits_{n=0}^{\infty }S_{n}\left( -A\right)
e_{1}^{2n}z^{n}}$, the right-hand side of the formula (2.2) can be expressed
as:%
\begin{eqnarray*}
\delta _{e_{1}e_{2}}f\left( e_{1}z\right) &=&\frac{\frac{e_{1}}{%
\dsum\limits_{n=0}^{\infty }S_{n}\left( -A\right) e_{1}^{2n}z^{n}}-\frac{%
e_{2}}{\dsum\limits_{n=0}^{\infty }S_{n}\left( -A\right) e_{2}^{2n}z^{n}}}{%
e_{1}-e_{2}} \\
&=&\frac{e_{1}\dsum\limits_{n=0}^{\infty }S_{n}\left( -A\right)
e_{2}^{2n}z^{n}-e_{2}\dsum\limits_{n=0}^{\infty }S_{n}\left( -A\right)
e_{1}^{2n}z^{n}}{\left( e_{1}-e_{2}\right) \left( \dsum\limits_{n=0}^{\infty
}S_{n}\left( -A\right) e_{1}^{2n}z^{n}\right) \left(
\dsum\limits_{n=0}^{\infty }S_{n}\left( -A\right) e_{2}^{2n}z^{n}\right) } \\
&=&\frac{\dsum\limits_{n=0}^{\infty }S_{n}\left( -A\right) \frac{%
e_{1}e_{2}^{2n}-e_{2}e_{1}^{2n}}{e_{1}-e_{2}}z^{n}}{\left(
\dsum\limits_{n=0}^{\infty }S_{n}\left( -A\right) e_{1}^{2n}z^{n}\right)
\left( \dsum\limits_{n=0}^{\infty }S_{n}\left( -A\right)
e_{2}^{2n}z^{n}\right) } \\
&=&\frac{S_{0}\left( -A\right) -e_{1}e_{2}S_{1}\left( -A\right)
z-S_{2}\left( -A\right) \frac{e_{1}^{3}-e_{2}^{3}}{e_{1}-e_{2}}z^{2}}{\left(
\dsum\limits_{n=0}^{\infty }S_{n}\left( -A\right) e_{1}^{2n}z^{n}\right)
\left( \dsum\limits_{n=0}^{\infty }S_{n}\left( -A\right)
e_{2}^{2n}z^{n}\right) } \\
&=&\frac{S_{0}\left( -A\right) -e_{1}e_{2}S_{1}\left( -A\right)
z-S_{2}\left( -A\right) e_{1}e_{2}\left( \left( e_{1}+e_{2}\right)
^{2}-e_{1}e_{2}\right) z^{2}}{\left( \dsum\limits_{n=0}^{\infty }S_{n}\left(
-A\right) e_{1}^{2n}z^{n}\right) \left( \dsum\limits_{n=0}^{\infty
}S_{n}\left( -A\right) e_{2}^{2n}z^{n}\right) } \\
&=&\frac{1+e_{1}e_{2}\left( a_{1}+a_{2}\right) z-e_{1}e_{2}a_{1}a_{2}\left(
\left( e_{1}+e_{2}\right) ^{2}-e_{1}e_{2}\right) z^{2}}{\left(
\dsum\limits_{n=0}^{\infty }S_{n}\left( -A\right) e_{1}^{2n}z^{n}\right)
\left( \dsum\limits_{n=0}^{\infty }S_{n}\left( -A\right)
e_{2}^{2n}z^{n}\right) }.
\end{eqnarray*}

Thus, this completes the proof.
\end{proof}

\begin{theorem}
Given two alphabets $A=\left\{ a_{1},a_{2}\right\} $ and $E=\left\{
e_{1},e_{2}\right\} ,$ we have:%
\begin{equation}
\sum\limits_{n=0}^{\infty }S_{n}\left( A\right) S_{2n+1}\left( E\right)
z^{n}=\frac{e_{1}+e_{2}-a_{1}a_{2}e_{1}^{2}e_{2}^{2}\left(
e_{1}+e_{2}\right) z^{2}}{\left( \dsum\limits_{n=0}^{\infty }S_{n}\left(
-A\right) e_{1}^{2n}z^{n}\right) \left( \dsum\limits_{n=0}^{\infty
}S_{n}\left( -A\right) e_{2}^{2n}z^{n}\right) }.  \tag{2.3}
\end{equation}
\end{theorem}

\begin{proof}
By applying the operator $\delta _{e_{1}e_{2}}^{2}\ $to the series$\ f\left(
e_{1}z\right) =\dsum\limits_{n=0}^{\infty }S_{n}\left( A\right)
e_{1}^{2n}z^{n}$, the left-hand side of the formula (2.3) can be written as:%
\begin{eqnarray*}
\delta _{e_{1}e_{2}}^{2}f\left( e_{1}z\right) &=&\frac{e_{1}^{2}\dsum%
\limits_{n=0}^{\infty }S_{n}\left( A\right)
e_{1}^{2n}z^{n}-e_{2}^{2}\dsum\limits_{n=0}^{\infty }S_{n}\left( A\right)
e_{2}^{2n}z^{n}}{e_{1}-e_{2}} \\
&=&\sum_{n=0}^{\infty }S_{n}\left( A\right) \left( \frac{%
e_{1}^{2n+2}-e_{2}^{2n+2}}{e_{1}-e_{2}}\right) z^{n} \\
&=&\sum_{n=0}^{\infty }S_{n}\left( A\right) S_{2n+1}\left( E\right) z^{n}.
\end{eqnarray*}

By applying the operator $\delta _{e_{1}e_{2}}^{2}\ $to the series$\ f\left(
e_{1}z\right) =\frac{1}{\dsum\limits_{n=0}^{\infty }S_{n}\left( -A\right)
e_{1}^{2n}z^{n}}$, the right-hand side of the formula (2.3) can be expressed
as:%
\begin{eqnarray*}
\delta _{e_{1}e_{2}}^{2}f\left( e_{1}z\right) &=&\frac{\frac{e_{1}^{2}}{%
\dsum\limits_{n=0}^{\infty }S_{n}\left( -A\right) e_{1}^{2n}z^{n}}-\frac{%
e_{2}^{2}}{\dsum\limits_{n=0}^{\infty }S_{n}\left( -A\right) e_{2}^{2n}z^{n}}%
}{e_{1}-e_{2}} \\
&=&\frac{e_{1}^{2}\dsum\limits_{n=0}^{\infty }S_{n}\left( -A\right)
e_{2}^{2n}z^{n}-e_{2}^{2}\dsum\limits_{n=0}^{\infty }S_{n}\left( -A\right)
e_{1}^{2n}z^{n}}{\left( e_{1}-e_{2}\right) \left( \dsum\limits_{n=0}^{\infty
}S_{n}\left( -A\right) e_{1}^{2n}z^{n}\right) \left(
\dsum\limits_{n=0}^{\infty }S_{n}\left( -A\right) e_{2}^{2n}z^{n}\right) } \\
&=&\frac{\dsum\limits_{n=0}^{\infty }S_{n}\left( -A\right) \frac{%
e_{1}^{2}e_{2}^{2n}-e_{2}^{2}e_{1}^{2n}}{e_{1}-e_{2}}z^{n}}{\left(
\dsum\limits_{n=0}^{\infty }S_{n}\left( -A\right) e_{1}^{2n}z^{n}\right)
\left( \dsum\limits_{n=0}^{\infty }S_{n}\left( -A\right)
e_{2}^{2n}z^{n}\right) } \\
&=&\frac{S_{0}\left( -A\right) \frac{e_{1}^{2}-e_{2}^{2}}{e_{1}-e_{2}}%
+S_{2}\left( -A\right) \frac{e_{1}^{2}e_{2}^{4}-e_{2}^{3}e_{1}^{4}}{%
e_{1}-e_{2}}z^{2}}{\left( \dsum\limits_{n=0}^{\infty }S_{n}\left( -A\right)
e_{1}^{2n}z^{n}\right) \left( \dsum\limits_{n=0}^{\infty }S_{n}\left(
-A\right) e_{2}^{2n}z^{n}\right) } \\
&=&\frac{S_{0}\left( -A\right) \left( e_{1}+e_{2}\right) -S_{2}\left(
-A\right) \left( e_{1}e_{2}\right) ^{2}\left( e_{1}+e_{2}\right) z^{2}}{%
\left( \dsum\limits_{n=0}^{\infty }S_{n}\left( -A\right)
e_{1}^{2n}z^{n}\right) \left( \dsum\limits_{n=0}^{\infty }S_{n}\left(
-A\right) e_{2}^{2n}z^{n}\right) } \\
&=&\frac{\left( e_{1}+e_{2}\right) -a_{1}a_{2}e_{1}^{2}e_{2}^{2}\left(
e_{1}+e_{2}\right) z^{2}}{\left( \dsum\limits_{n=0}^{\infty }S_{n}\left(
-A\right) e_{1}^{2n}z^{n}\right) \left( \dsum\limits_{n=0}^{\infty
}S_{n}\left( -A\right) e_{2}^{2n}z^{n}\right) }.
\end{eqnarray*}

Thus, this completes the proof.
\end{proof}

Based on the relationship (2.3) we get:%
\begin{equation}
\sum\limits_{n=0}^{\infty }S_{n-1}\left( A\right) S_{2n-1}\left( E\right)
z^{n}=\frac{\left( e_{1}+e_{2}\right) z-a_{1}a_{2}e_{1}^{2}e_{2}^{2}\left(
e_{1}+e_{2}\right) z^{3}}{\left( \dsum\limits_{n=0}^{\infty }S_{n}\left(
-A\right) e_{1}^{2n}z^{n}\right) \left( \dsum\limits_{n=0}^{\infty
}S_{n}\left( -A\right) e_{2}^{2n}z^{n}\right) }.  \tag{2.4}
\end{equation}

\begin{theorem}
Given two alphabets $A=\left\{ a_{1},a_{2}\right\} $ and $E=\left\{
e_{1},e_{2}\right\} ,$ we have:%
\begin{equation}
\sum\limits_{n=0}^{\infty }S_{n}\left( A\right) S_{2n+2}\left( E\right)
z^{n}=\frac{\left( e_{1}+e_{2}\right)
^{2}-e_{1}e_{2}-e_{1}^{2}e_{2}^{2}\left( a_{1}+a_{2}\right)
z-a_{1}a_{2}e_{1}^{3}e_{2}^{3}z^{2}}{\left( \dsum\limits_{n=0}^{\infty
}S_{n}\left( -A\right) e_{1}^{2n}z^{n}\right) \left(
\dsum\limits_{n=0}^{\infty }S_{n}\left( -A\right) e_{2}^{2n}z^{n}\right) }. 
\tag{2.5}
\end{equation}
\end{theorem}

\begin{proof}
The proof is similar to the proof of Theorem 3, but now using the operator $%
\delta _{e_{1}e_{2}}^{3}$ to the series $\dsum\limits_{n=0}^{\infty
}S_{n}\left( A\right) e_{1}^{2n}z^{n}=\frac{1}{\dsum\limits_{n=0}^{\infty
}S_{n}\left( -A\right) e_{1}^{2n}z^{n}}.$
\end{proof}

Note that, based on the relationship (2.5), we get:%
\begin{equation}
\sum\limits_{n=0}^{\infty }S_{n-1}\left( A\right) S_{2n}\left( E\right)
z^{n}=\frac{\left( \left( e_{1}+e_{2}\right) ^{2}-e_{1}e_{2}\right)
z-e_{1}^{2}e_{2}^{2}\left( a_{1}+a_{2}\right)
z^{2}-a_{1}a_{2}e_{1}^{3}e_{2}^{3}z^{3}}{\left( \dsum\limits_{n=0}^{\infty
}S_{n}\left( -A\right) e_{1}^{2n}z^{n}\right) \left(
\dsum\limits_{n=0}^{\infty }S_{n}\left( -A\right) e_{2}^{2n}z^{n}\right) }. 
\tag{2.6}
\end{equation}

\begin{theorem}
Given two alphabets $A=\left\{ a_{1},a_{2}\right\} $ and $E=\left\{
e_{1},e_{2}\right\} ,$ we have:%
\begin{equation}
\sum\limits_{n=0}^{\infty }S_{n}\left( A\right) S_{2n+3}\left( E\right)
z^{n}=\frac{\left( e_{1}+e_{2}\right) \left( \left( e_{1}+e_{2}\right)
^{2}-2e_{1}e_{2}\right) -e_{1}^{2}e_{2}^{2}\left( a_{1}+a_{2}\right) \left(
e_{1}+e_{2}\right) z}{\left( \dsum\limits_{n=0}^{\infty }S_{n}\left(
-A\right) e_{1}^{2n}z^{n}\right) \left( \dsum\limits_{n=0}^{\infty
}S_{n}\left( -A\right) e_{2}^{2n}z^{n}\right) }.  \tag{2.7}
\end{equation}
\end{theorem}

\begin{proof}
The proof is similar to the proof of Theorem 3, but now using the operator $%
\delta _{e_{1}e_{2}}^{4}$ to the series $\dsum\limits_{n=0}^{\infty
}S_{n}\left( A\right) e_{1}^{2n}z^{n}=\frac{1}{\dsum\limits_{n=0}^{\infty
}S_{n}\left( -A\right) e_{1}^{2n}z^{n}}.$
\end{proof}

From (2.7), we get:%
\begin{equation}
\sum\limits_{n=0}^{\infty }S_{n-1}\left( A\right) S_{2n+1}\left( E\right)
z^{n}=\frac{\left( e_{1}+e_{2}\right) \left( \left( e_{1}+e_{2}\right)
^{2}-2e_{1}e_{2}\right) z-e_{1}^{2}e_{2}^{2}\left( a_{1}+a_{2}\right) \left(
e_{1}+e_{2}\right) z^{2}}{\left( \dsum\limits_{n=0}^{\infty }S_{n}\left(
-A\right) e_{1}^{2n}z^{n}\right) \left( \dsum\limits_{n=0}^{\infty
}S_{n}\left( -A\right) e_{2}^{2n}z^{n}\right) }.  \tag{2.8}
\end{equation}

\section{\textbf{Generating functions of odd and even some }$\left( 
\boldsymbol{p,q}\right) $\textbf{-numbers}}

In this part, we now derive the new generating functions for odd and even
phrases\emph{\ }of\emph{\ }some special numbers with parameters $p$ and $q$.
By making use of the Theorems 1, 2 and 3 we investigate some special cases ($%
A=\left\{ 1,0\right\} $ and $E=\left\{ e_{1},-e_{2}\right\} $) as follows:

\begin{lemma}
Given an alphabet $E=\left\{ e_{1},-e_{2}\right\} ,$ then we have:%
\begin{equation}
\sum\limits_{n=0}^{\infty }S_{2n-1}\left( e_{1}+\left[ -e_{2}\right] \right)
z^{n}=\frac{\left( e_{1}-e_{2}\right) z}{1-\left( \left( e_{1}-e_{2}\right)
^{2}+2e_{1}e_{2}\right) z+e_{1}^{2}e_{2}^{2}z^{2}}.  \tag{3.1}
\end{equation}
\end{lemma}

\begin{lemma}
Given an alphabet $E=\left\{ e_{1},-e_{2}\right\} ,$ then we have:%
\begin{equation}
\sum\limits_{n=0}^{\infty }S_{2n}\left( e_{1}+\left[ -e_{2}\right] \right)
z^{n}=\frac{1-e_{1}e_{2}z}{1-\left( \left( e_{1}-e_{2}\right)
^{2}+2e_{1}e_{2}\right) z+e_{1}^{2}e_{2}^{2}z^{2}}.  \tag{3.2}
\end{equation}
\end{lemma}

\begin{lemma}
Given an alphabet $E=\left\{ e_{1},-e_{2}\right\} ,$ then we have:%
\begin{equation}
\sum\limits_{n=0}^{\infty }S_{2n+1}\left( e_{1}+\left[ -e_{2}\right] \right)
z^{n}=\frac{e_{1}-e_{2}}{1-\left( \left( e_{1}-e_{2}\right)
^{2}+2e_{1}e_{2}\right) z+e_{1}^{2}e_{2}^{2}z^{2}}.  \tag{3.3}
\end{equation}
\end{lemma}

This part consists of three cases.

\textbf{Case 1. }The substitution of $\left\{ 
\begin{array}{l}
e_{1}-e_{2}=p \\ 
e_{1}e_{2}=q%
\end{array}%
\right. $ in (3.1), (3.2) and (3.3), we obtain:%
\begin{equation}
\sum\limits_{n=0}^{\infty }S_{2n-1}\left( e_{1}+\left[ -e_{2}\right] \right)
z^{n}=\frac{pz}{1-\left( p^{2}+2q\right) z+q^{2}z^{2}},  \tag{3.4}
\end{equation}%
\begin{equation}
\sum\limits_{n=0}^{\infty }S_{2n}\left( e_{1}+\left[ -e_{2}\right] \right)
z^{n}=\frac{1-qz}{1-\left( p^{2}+2q\right) z+q^{2}z^{2}},  \tag{3.5}
\end{equation}%
\begin{equation}
\sum\limits_{n=0}^{\infty }S_{2n+1}\left( e_{1}+\left[ -e_{2}\right] \right)
z^{n}=\frac{p}{1-\left( p^{2}+2q\right) z+q^{2}z^{2}},  \tag{3.6}
\end{equation}%
respectively, and we have the following propositions and theorems.

\begin{proposition}
For $n\in 
\mathbb{N}
,$ the new generating function of even $\left( p,q\right) $-Fibonacci
numbers is given by:%
\begin{equation}
\sum\limits_{n=0}^{\infty }F_{p,q,2n}z^{n}=\frac{pz}{1-\left(
p^{2}+2q\right) z+q^{2}z^{2}}.  \tag{3.7}
\end{equation}
\end{proposition}

\begin{proof}
By \cite{NABIHA}, we have%
\begin{equation*}
F_{p,q,n}=S_{n-1}\left( e_{1}+\left[ -e_{2}\right] \right) .
\end{equation*}

Writing $\left( 2n\right) $ instead of $\left( n\right) $, we get%
\begin{equation*}
F_{p,q,2n}=S_{2n-1}\left( e_{1}+\left[ -e_{2}\right] \right) .
\end{equation*}

Then, by relation (3.4), we obtain%
\begin{eqnarray*}
\sum\limits_{n=0}^{\infty }F_{p,q,2n}z^{n} &=&\sum\limits_{n=0}^{\infty
}S_{2n-1}\left( e_{1}+\left[ -e_{2}\right] \right) z^{n} \\
&=&\frac{pz}{1-\left( p^{2}+2q\right) z+q^{2}z^{2}}.
\end{eqnarray*}

Hence, we obtain the desired result.
\end{proof}

\begin{proposition}
For $n\in 
\mathbb{N}
,$ the new generating function of odd $\left( p,q\right) $-Fibonacci numbers
is given by:%
\begin{equation}
\sum\limits_{n=0}^{\infty }F_{p,q,2n+1}z^{n}=\frac{1-qz}{1-\left(
p^{2}+2q\right) z+q^{2}z^{2}}.  \tag{3.8}
\end{equation}
\end{proposition}

\begin{proof}
By referred to \cite{NABIHA}, we have%
\begin{equation*}
F_{p,q,n}=S_{n-1}\left( e_{1}+\left[ -e_{2}\right] \right) .
\end{equation*}%
Substituting $n$ by $\left( 2n+1\right) $, we obtain%
\begin{equation*}
F_{p,q,2n+1}=S_{2n}\left( e_{1}+\left[ -e_{2}\right] \right) .
\end{equation*}

From the relation (3.5), we can write 
\begin{eqnarray*}
\sum\limits_{n=0}^{\infty }F_{p,q,2n+1}z^{n} &=&\sum\limits_{n=0}^{\infty
}S_{2n}\left( e_{1}+\left[ -e_{2}\right] \right) z^{n} \\
&=&\frac{1-qz}{1-\left( p^{2}+2q\right) z+q^{2}z^{2}}.
\end{eqnarray*}

As required.
\end{proof}

\begin{theorem}
For $n\in 
\mathbb{N}
,$ the new generating function of even $\left( p,q\right) $-Lucas numbers is
given by:%
\begin{equation}
\sum\limits_{n=0}^{\infty }L_{p,q,2n}z^{n}=\frac{2-\left( p^{2}+2q\right) z}{%
1-\left( p^{2}+2q\right) z+q^{2}z^{2}}.  \tag{3.9}
\end{equation}
\end{theorem}

\begin{proof}
We have%
\begin{equation*}
L_{p,q,n}=2S_{n}\left( e_{1}+\left[ -e_{2}\right] \right) -pS_{n-1}\left(
e_{1}+\left[ -e_{2}\right] \right) ,\text{ (see \cite{NABIHA}).}
\end{equation*}%
By setting $n=2n,$ we get%
\begin{equation*}
L_{p,q,2n}=2S_{2n}\left( e_{1}+\left[ -e_{2}\right] \right) -pS_{2n-1}\left(
e_{1}+\left[ -e_{2}\right] \right) .
\end{equation*}%
Then%
\begin{eqnarray*}
\sum\limits_{n=0}^{\infty }L_{p,q,2n}z^{n} &=&\sum\limits_{n=0}^{\infty
}\left( 2S_{2n}\left( e_{1}+\left[ -e_{2}\right] \right) -pS_{2n-1}\left(
e_{1}+\left[ -e_{2}\right] \right) \right) z^{n} \\
&=&2\sum\limits_{n=0}^{\infty }S_{2n}\left( e_{1}+\left[ -e_{2}\right]
\right) z^{n}-p\sum\limits_{n=0}^{\infty }S_{2n-1}\left( e_{1}+\left[ -e_{2}%
\right] \right) z^{n}.
\end{eqnarray*}

Multiplying the equation (3.5) by $\left( 2\right) $ and adding it to the
equation obtained by (3.4) multiplying by $(-p)$, then we obtain the
following equality%
\begin{eqnarray*}
\sum\limits_{n=0}^{\infty }L_{p,q,2n}z^{n} &=&\frac{2\left( 1-qz\right) }{%
1-\left( p^{2}+2q\right) z+q^{2}z^{2}}-\frac{p^{2}z}{1-\left(
p^{2}+2q\right) z+q^{2}z^{2}} \\
&=&\frac{2-\left( p^{2}+2q\right) z}{1-\left( p^{2}+2q\right) z+q^{2}z^{2}}.
\end{eqnarray*}

Which completes the proof.
\end{proof}

\begin{theorem}
For $n\in 
\mathbb{N}
,$ the new generating function of odd $\left( p,q\right) $-Lucas numbers is
given by:%
\begin{equation}
\sum\limits_{n=0}^{\infty }L_{p,q,2n+1}z^{n}=\frac{p+pqz}{1-\left(
p^{2}+2q\right) z+q^{2}z^{2}}.  \tag{3.10}
\end{equation}
\end{theorem}

\begin{proof}
Recall that, we have $L_{p,q,n}=2S_{n}\left( e_{1}+\left[ -e_{2}\right]
\right) -pS_{n-1}\left( e_{1}+\left[ -e_{2}\right] \right) ,$ (see \cite%
{NABIHA}).

By putting $n=2n+1,$ we get%
\begin{equation*}
L_{p,q,2n+1}=2S_{2n+1}\left( e_{1}+\left[ -e_{2}\right] \right)
-pS_{2n}\left( e_{1}+\left[ -e_{2}\right] \right) \text{.}
\end{equation*}%
Then%
\begin{eqnarray*}
\sum\limits_{n=0}^{\infty }L_{p,q,2n+1}z^{n} &=&\sum\limits_{n=0}^{\infty
}\left( 2S_{2n+1}\left( e_{1}+\left[ -e_{2}\right] \right) -pS_{2n}\left(
e_{1}+\left[ -e_{2}\right] \right) \right) z^{n} \\
&=&2\sum\limits_{n=0}^{\infty }S_{2n+1}\left( e_{1}+\left[ -e_{2}\right]
\right) z^{n}-p\sum\limits_{n=0}^{\infty }S_{2n}\left( e_{1}+\left[ -e_{2}%
\right] \right) z^{n}.
\end{eqnarray*}

Multiplying the equation (3.6) by $\left( 2\right) $ and adding it to the
equation obtained by (3.5) multiplying by $(-p)$, then we obtain the
following equality%
\begin{eqnarray*}
\sum\limits_{n=0}^{\infty }L_{p,q,2n+1}z^{n} &=&\frac{2p}{1-\left(
p^{2}+2q\right) z+q^{2}z^{2}}-\frac{p^{2}\left( 1-qz\right) }{1-\left(
p^{2}+2q\right) z+q^{2}z^{2}} \\
&=&\frac{p+pqz}{1-\left( p^{2}+2q\right) z+q^{2}z^{2}}.
\end{eqnarray*}

Thus, this completes the proof.
\end{proof}

\begin{itemize}
\item Taking $p=k$ and $q=1$ in the Eqs. (3.7)-(3.10),\ we get the new
generating functions of even and odd $k$-Fibonacci and $k$-Lucas numbers.
The calculation and results are listed in the Tab. 2.%
\begin{equation*}
\begin{tabular}{|c|c|}
\hline
C$\text{oefficient of }z^{n}$ & Generating function \\ \hline
$F_{k,2n}$ & $\frac{kz}{1-\left( k^{2}+2\right) z+z^{2}}$ \\ \hline
$F_{k,2n+1}$ & $\frac{1-z}{1-\left( k^{2}+2\right) z+z^{2}}$ \\ \hline
$L_{k,2n}$ & $\frac{2-\left( k^{2}+2\right) z}{1-\left( k^{2}+2\right)
z+z^{2}}$ \\ \hline
$L_{k,2n+1}$ & $\frac{k+kz}{1-\left( k^{2}+2\right) z+z^{2}}$ \\ \hline
\end{tabular}%
\end{equation*}
\end{itemize}

\begin{center}
\textbf{Table 2.} New generating functions for even and odd $k$-Fibonacci
and $k$-Lucas numbers.
\end{center}

\begin{remark}
If we take $k=1$ in the Tab. 2,\ we get the generating functions of even and
odd Fibonacci and Lucas numbers given in the article of Mezo \cite{MEZO} as
follows:%
\begin{equation*}
\sum\limits_{n=0}^{\infty }F_{2n}z^{n}=\frac{z}{1-3z+z^{2}},\text{ }%
\sum\limits_{n=0}^{\infty }F_{2n+1}z^{n}=\frac{1-z}{1-3z+z^{2}},\text{ }%
\sum\limits_{n=0}^{\infty }L_{2n}z^{n}=\frac{2-3z}{1-3z+z^{2}}\text{ and }%
\sum\limits_{n=0}^{\infty }L_{2n+1}z^{n}=\frac{1+z}{1-3z+z^{2}}.
\end{equation*}
\end{remark}

\textbf{Case 2. }The substitution of $\left\{ 
\begin{array}{l}
e_{1}-e_{2}=2p \\ 
e_{1}e_{2}=q%
\end{array}%
\right. $ in (3.1), (3.2) and (3.3), we obtain:%
\begin{equation}
\sum\limits_{n=0}^{\infty }S_{2n-1}\left( e_{1}+\left[ -e_{2}\right] \right)
z^{n}=\frac{2pz}{1-2\left( 2p^{2}+q\right) z+q^{2}z^{2}},  \tag{3.11}
\end{equation}%
\begin{equation}
\sum\limits_{n=0}^{\infty }S_{2n}\left( e_{1}+\left[ -e_{2}\right] \right)
z^{n}=\frac{1-qz}{1-2\left( 2p^{2}+q\right) z+q^{2}z^{2}},  \tag{3.12}
\end{equation}%
\begin{equation}
\sum\limits_{n=0}^{\infty }S_{2n+1}\left( e_{1}+\left[ -e_{2}\right] \right)
z^{n}=\frac{2p}{1-2\left( 2p^{2}+q\right) z+q^{2}z^{2}},  \tag{3.13}
\end{equation}%
respectively, and we have the following propositions and theorems.

\begin{proposition}
For $n\in 
\mathbb{N}
,$ the new generating function of even $\left( p,q\right) $-Pell numbers is
given by:%
\begin{equation}
\sum\limits_{n=0}^{\infty }P_{p,q,2n}z^{n}=\frac{2pz}{1-2\left(
2p^{2}+q\right) z+q^{2}z^{2}}.  \tag{3.14}
\end{equation}
\end{proposition}

\begin{proof}
Recall that, we have 
\begin{equation*}
P_{p,q,n}=S_{n-1}\left( e_{1}+\left[ -e_{2}\right] \right) ,\text{ (see \cite%
{NABIHA}).}
\end{equation*}

Writing $\left( 2n\right) $ instead of $\left( n\right) $, we obtain%
\begin{equation*}
P_{p,q,2n}=S_{2n-1}\left( e_{1}+\left[ -e_{2}\right] \right) .
\end{equation*}

Therefore, according to the relationship (3.11), we get 
\begin{eqnarray*}
\sum\limits_{n=0}^{\infty }P_{p,q,2n}z^{n} &=&\sum\limits_{n=0}^{\infty
}S_{2n-1}\left( e_{1}+\left[ -e_{2}\right] \right) z^{n} \\
&=&\frac{2pz}{1-2\left( 2p^{2}+q\right) z+q^{2}z^{2}}.
\end{eqnarray*}

Hence, we obtain the desired result.
\end{proof}

\begin{proposition}
For $n\in 
\mathbb{N}
,$ the new generating function of odd $\left( p,q\right) $-Pell numbers is
given by:%
\begin{equation}
\sum\limits_{n=0}^{\infty }P_{p,q,2n+1}z^{n}=\frac{1-qz}{1-2\left(
2p^{2}+q\right) z+q^{2}z^{2}}.  \tag{3.15}
\end{equation}
\end{proposition}

\begin{proof}
By referred to \cite{NABIHA}, we have%
\begin{equation*}
P_{p,q,n}=S_{n-1}\left( e_{1}+\left[ -e_{2}\right] \right) .
\end{equation*}

Substituting $n$ by $\left( 2n+1\right) $, we obtain%
\begin{equation*}
P_{p,q,2n+1}=S_{2n}\left( e_{1}+\left[ -e_{2}\right] \right) .
\end{equation*}

From the relation (3.12), we can write 
\begin{eqnarray*}
\sum\limits_{n=0}^{\infty }P_{p,q,2n+1}z^{n} &=&\sum\limits_{n=0}^{\infty
}S_{2n}\left( e_{1}+\left[ -e_{2}\right] \right) z^{n} \\
&=&\frac{1-qz}{1-2\left( 2p^{2}+q\right) z+q^{2}z^{2}}.
\end{eqnarray*}

As required.
\end{proof}

\begin{theorem}
For $n\in 
\mathbb{N}
,$ the new generating function of even $\left( p,q\right) $-Pell Lucas
numbers is given by:%
\begin{equation}
\sum\limits_{n=0}^{\infty }Q_{p,q,2n}z^{n}=\frac{2-2\left( 2p^{2}+q\right) z%
}{1-2\left( 2p^{2}+q\right) z+q^{2}z^{2}}.  \tag{3.16}
\end{equation}
\end{theorem}

\begin{proof}
Once more, by \cite{NABIHA} we have%
\begin{equation*}
Q_{p,q,n}=2S_{n}\left( e_{1}+\left[ -e_{2}\right] \right) -2pS_{n-1}\left(
e_{1}+\left[ -e_{2}\right] \right) \text{.}
\end{equation*}%
By setting $n=2n,$ we get%
\begin{equation*}
Q_{p,q,2n}=2S_{2n}\left( e_{1}+\left[ -e_{2}\right] \right)
-2pS_{2n-1}\left( e_{1}+\left[ -e_{2}\right] \right) \text{.}
\end{equation*}%
Then%
\begin{eqnarray*}
\sum\limits_{n=0}^{\infty }Q_{p,q,2n}z^{n} &=&\sum\limits_{n=0}^{\infty
}\left( 2S_{2n}\left( e_{1}+\left[ -e_{2}\right] \right) -2pS_{2n-1}\left(
e_{1}+\left[ -e_{2}\right] \right) \right) z^{n} \\
&=&2\sum\limits_{n=0}^{\infty }S_{2n}\left( e_{1}+\left[ -e_{2}\right]
\right) z^{n}-2p\sum\limits_{n=0}^{\infty }S_{2n-1}\left( e_{1}+\left[ -e_{2}%
\right] \right) z^{n}.
\end{eqnarray*}

Multiplying the equation (3.12) by $\left( 2\right) $ and adding it to the
equation obtained by (3.11) multiplying by $(-2p)$, then we obtain the
following equality%
\begin{eqnarray*}
\sum\limits_{n=0}^{\infty }Q_{p,q,2n}z^{n} &=&\frac{2\left( 1-qz\right) }{%
1-2\left( 2p^{2}+q\right) z+q^{2}z^{2}}-\frac{4p^{2}z}{1-2\left(
2p^{2}+q\right) z+q^{2}z^{2}} \\
&=&\frac{2-2\left( 2p^{2}+q\right) z}{1-2\left( 2p^{2}+q\right) z+q^{2}z^{2}}%
.
\end{eqnarray*}%
Which completes the proof.
\end{proof}

\begin{theorem}
For $n\in 
\mathbb{N}
,$ the new generating function of odd $\left( p,q\right) $-Pell Lucas
numbers is given by:%
\begin{equation}
\sum\limits_{n=0}^{\infty }Q_{p,q,2n+1}z^{n}=\frac{2p+2pqz}{1-2\left(
2p^{2}+q\right) z+q^{2}z^{2}}.  \tag{3.17}
\end{equation}
\end{theorem}

\begin{proof}
Recall that, we have $Q_{p,q,n}=2S_{n}\left( e_{1}+\left[ -e_{2}\right]
\right) -2pS_{n-1}\left( e_{1}+\left[ -e_{2}\right] \right) ,$ (see \cite%
{NABIHA}).

By putting $n=2n+1,$ we get%
\begin{equation*}
Q_{p,q,2n+1}=2S_{2n+1}\left( e_{1}+\left[ -e_{2}\right] \right)
-2pS_{2n}\left( e_{1}+\left[ -e_{2}\right] \right) .
\end{equation*}%
Then%
\begin{eqnarray*}
\sum\limits_{n=0}^{\infty }Q_{p,q,2n+1}z^{n} &=&\sum\limits_{n=0}^{\infty
}\left( 2S_{2n+1}\left( e_{1}+\left[ -e_{2}\right] \right) -2pS_{2n}\left(
e_{1}+\left[ -e_{2}\right] \right) \right) z^{n} \\
&=&2\sum\limits_{n=0}^{\infty }S_{2n+1}\left( e_{1}+\left[ -e_{2}\right]
\right) z^{n}-2p\sum\limits_{n=0}^{\infty }S_{2n}\left( e_{1}+\left[ -e_{2}%
\right] \right) z^{n}.
\end{eqnarray*}

Multiplying the equation (3.13) by $\left( 2\right) $ and adding it to the
equation obtained by (3.12) multiplying by $(-2p)$, then we obtain the
following equality%
\begin{eqnarray*}
\sum\limits_{n=0}^{\infty }Q_{p,q,2n+1}z^{n} &=&\frac{4p}{1-2\left(
2p^{2}+q\right) z+q^{2}z^{2}}-\frac{2p\left( 1-qz\right) }{1-2\left(
2p^{2}+q\right) z+q^{2}z^{2}} \\
&=&\frac{2p+2pqz}{1-2\left( 2p^{2}+q\right) z+q^{2}z^{2}}.
\end{eqnarray*}%
Thus, this completes the proof.
\end{proof}

\begin{itemize}
\item Taking $p=1$ and $q=k$ in the Eqs. (3.14)-(3.17),\ we get the new
generating functions of even and odd $k$-Pell and $k$-Pell Lucas numbers.
The calculation and results are listed in the Tab. 3.%
\begin{equation*}
\begin{tabular}{|c|c|}
\hline
C$\text{oefficient of }z^{n}$ & Generating function \\ \hline
$P_{k,2n}$ & $\frac{2z}{1-2\left( k+2\right) z+k^{2}z^{2}}$ \\ \hline
$P_{k,2n+1}$ & $\frac{1-kz}{1-2\left( k+2\right) z+k^{2}z^{2}}$ \\ \hline
$Q_{k,2n}$ & $\frac{2-2\left( k+2\right) z}{1-2\left( k+2\right) z+k^{2}z^{2}%
}$ \\ \hline
$Q_{k,2n+1}$ & $\frac{2+2kz}{1-2\left( k+2\right) z+k^{2}z^{2}}$ \\ \hline
\end{tabular}%
\end{equation*}
\end{itemize}

\begin{center}
\textbf{Table 3.} New generating functions for even and odd $k$-Pell and $k$%
-Pell Lucas numbers.
\end{center}

\begin{remark}
If we take $k=1$ in the Tab. 3,\ we get the generating functions of even and
odd Pell and Pell Lucas numbers given in the article of Mezo \cite{MEZO} as
follows:%
\begin{equation*}
\sum\limits_{n=0}^{\infty }P_{2n}z^{n}=\frac{2z}{1-6z+z^{2}},\text{ }%
\sum\limits_{n=0}^{\infty }P_{2n+1}z^{n}=\frac{1-z}{1-6z+z^{2}},\text{ }%
\sum\limits_{n=0}^{\infty }Q_{2n}z^{n}=\frac{2-6z}{1-6z+z^{2}}\text{ and }%
\sum\limits_{n=0}^{\infty }Q_{2n+1}z^{n}=\frac{2+2z}{1-6z+z^{2}}.
\end{equation*}
\end{remark}

\textbf{Case 3. }The substitution of $\left\{ 
\begin{array}{l}
e_{1}-e_{2}=p \\ 
e_{1}e_{2}=2q%
\end{array}%
\right. $ in (3.1), (3.2) and (3.3), we obtain:%
\begin{equation}
\sum\limits_{n=0}^{\infty }S_{2n-1}\left( e_{1}+\left[ -e_{2}\right] \right)
z^{n}=\frac{pz}{1-\left( p^{2}+4q\right) z+4q^{2}z^{2}},  \tag{3.18}
\end{equation}%
\begin{equation}
\sum\limits_{n=0}^{\infty }S_{2n}\left( e_{1}+\left[ -e_{2}\right] \right)
z^{n}=\frac{1-2qz}{1-\left( p^{2}+4q\right) z+4q^{2}z^{2}},  \tag{3.19}
\end{equation}%
\begin{equation}
\sum\limits_{n=0}^{\infty }S_{2n+1}\left( e_{1}+\left[ -e_{2}\right] \right)
z^{n}=\frac{p}{1-\left( p^{2}+4q\right) z+4q^{2}z^{2}},  \tag{3.20}
\end{equation}%
respectively, and we have the following propositions and theorems.

\begin{proposition}
For $n\in 
\mathbb{N}
,$ the new generating function of even $\left( p,q\right) $-Jacobsthal
numbers is given by:%
\begin{equation}
\sum\limits_{n=0}^{\infty }J_{p,q,2n}z^{n}=\frac{pz}{1-\left(
p^{2}+4q\right) z+4q^{2}z^{2}}.  \tag{3.21}
\end{equation}
\end{proposition}

\begin{proof}
By \cite{NABIHA}, we have%
\begin{equation*}
J_{p,q,n}=S_{n-1}\left( e_{1}+\left[ -e_{2}\right] \right) .
\end{equation*}

Writing $\left( 2n\right) $ instead of $\left( n\right) $, we get%
\begin{equation*}
J_{p,q,2n}=S_{2n-1}\left( e_{1}+\left[ -e_{2}\right] \right) .
\end{equation*}

Then, by relation (3.18), we obtain 
\begin{eqnarray*}
\sum\limits_{n=0}^{\infty }J_{p,q,2n}z^{n} &=&\sum\limits_{n=0}^{\infty
}S_{2n-1}\left( e_{1}+\left[ -e_{2}\right] \right) z^{n} \\
&=&\frac{pz}{1-\left( p^{2}+4q\right) z+4q^{2}z^{2}}.
\end{eqnarray*}

Hence, we obtain the desired result.
\end{proof}

\begin{proposition}
For $n\in 
\mathbb{N}
,$ the new generating function of odd $\left( p,q\right) $-Jacobsthal
numbers is given by:%
\begin{equation}
\sum\limits_{n=0}^{\infty }J_{p,q,2n+1}z^{n}=\frac{1-2qz}{1-\left(
p^{2}+4q\right) z+4q^{2}z^{2}}.  \tag{3.22}
\end{equation}
\end{proposition}

\begin{proof}
By referred to \cite{NABIHA}, we have%
\begin{equation*}
J_{p,q,n}=S_{n-1}\left( e_{1}+\left[ -e_{2}\right] \right) .
\end{equation*}

Substituting\ $n$\ by $\left( 2n+1\right) $, we obtain%
\begin{equation*}
J_{p,q,2n+1}=S_{2n}\left( e_{1}+\left[ -e_{2}\right] \right) .
\end{equation*}

From the relation (3.19), we can write%
\begin{eqnarray*}
\sum\limits_{n=0}^{\infty }J_{p,q,2n+1}z^{n} &=&\sum\limits_{n=0}^{\infty
}S_{2n}\left( e_{1}+\left[ -e_{2}\right] \right) z^{n} \\
&=&\frac{1-2qz}{1-\left( p^{2}+4q\right) z+4q^{2}z^{2}}.
\end{eqnarray*}

As required.
\end{proof}

\begin{theorem}
For $n\in 
\mathbb{N}
,$ the new generating function of even $\left( p,q\right) $-Jacobsthal Lucas
numbers is given by:%
\begin{equation}
\sum\limits_{n=0}^{\infty }j_{p,q,2n}z^{n}=\frac{2-\left( p^{2}+4q\right) z}{%
1-\left( p^{2}+4q\right) z+4q^{2}z^{2}}.  \tag{3.23}
\end{equation}
\end{theorem}

\begin{proof}
We have%
\begin{equation*}
j_{p,q,n}=2S_{n}\left( e_{1}+\left[ -e_{2}\right] \right) -pS_{n-1}\left(
e_{1}+\left[ -e_{2}\right] \right) ,\text{ (see \cite{NABIHA}).}
\end{equation*}%
By setting $n=2n,$ we get%
\begin{equation*}
j_{p,q,2n}=2S_{2n}\left( e_{1}+\left[ -e_{2}\right] \right) -pS_{2n-1}\left(
e_{1}+\left[ -e_{2}\right] \right) \text{.}
\end{equation*}%
Then%
\begin{eqnarray*}
\sum\limits_{n=0}^{\infty }j_{p,q,2n}z^{n} &=&\sum\limits_{n=0}^{\infty
}\left( 2S_{2n}\left( e_{1}+\left[ -e_{2}\right] \right) -pS_{2n-1}\left(
e_{1}+\left[ -e_{2}\right] \right) \right) z^{n} \\
&=&2\sum\limits_{n=0}^{\infty }S_{2n}\left( e_{1}+\left[ -e_{2}\right]
\right) z^{n}-p\sum\limits_{n=0}^{\infty }S_{2n-1}\left( e_{1}+\left[ -e_{2}%
\right] \right) z^{n}.
\end{eqnarray*}

Multiplying the equation (3.19) by $\left( 2\right) $ and adding it to the
equation obtained by (3.18) multiplying by $(-p)$, then we obtain the
following equality%
\begin{eqnarray*}
\sum\limits_{n=0}^{\infty }j_{p,q,2n}z^{n} &=&\frac{2\left( 1-2qz\right) }{%
1-\left( p^{2}+4q\right) z+4q^{2}z^{2}}-\frac{p^{2}z}{1-\left(
p^{2}+4q\right) z+4q^{2}z^{2}} \\
&=&\frac{2-\left( p^{2}+4q\right) z}{1-\left( p^{2}+4q\right) z+4q^{2}z^{2}}.
\end{eqnarray*}

Which completes the proof.
\end{proof}

\begin{theorem}
For $n\in 
\mathbb{N}
,$ the new generating function of odd $\left( p,q\right) $-Jacobsthal Lucas
numbers is given by:%
\begin{equation}
\sum\limits_{n=0}^{\infty }j_{p,q,2n+1}z^{n}=\frac{p+2pqz}{1-\left(
p^{2}+4q\right) z+4q^{2}z^{2}}.  \tag{3.24}
\end{equation}
\end{theorem}

\begin{proof}
By \cite{NABIHA} we have 
\begin{equation*}
j_{p,q,n}=2S_{n}\left( e_{1}+\left[ -e_{2}\right] \right) -pS_{n-1}\left(
e_{1}+\left[ -e_{2}\right] \right) .
\end{equation*}

By putting $n=2n+1,$ we get%
\begin{equation*}
j_{p,q,2n+1}=2S_{2n+1}\left( e_{1}+\left[ -e_{2}\right] \right)
-pS_{2n}\left( e_{1}+\left[ -e_{2}\right] \right) .
\end{equation*}%
Then%
\begin{eqnarray*}
\sum\limits_{n=0}^{\infty }j_{p,q,2n+1}z^{n} &=&\sum\limits_{n=0}^{\infty
}\left( 2S_{2n+1}\left( e_{1}+\left[ -e_{2}\right] \right) -pS_{2n}\left(
e_{1}+\left[ -e_{2}\right] \right) \right) z^{n} \\
&=&2\sum\limits_{n=0}^{\infty }S_{2n+1}\left( e_{1}+\left[ -e_{2}\right]
\right) -p\sum\limits_{n=0}^{\infty }S_{2n}\left( e_{1}+\left[ -e_{2}\right]
\right) \text{.}
\end{eqnarray*}

Multiplying the equation (3.20) by $\left( 2\right) $ and adding it to the
equation obtained by (3.19) multiplying by $(-p)$, then we obtain the
following equality%
\begin{eqnarray*}
\sum\limits_{n=0}^{\infty }j_{p,q,2n+1}z^{n} &=&\frac{2p}{1-\left(
p^{2}+4q\right) z+4q^{2}z^{2}}-\frac{p\left( 1-2qz\right) }{1-\left(
p^{2}+4q\right) z+4q^{2}z^{2}} \\
&=&\frac{p+2pqz}{1-\left( p^{2}+4q\right) z+4q^{2}z^{2}}.
\end{eqnarray*}%
Thus, this completes the proof.
\end{proof}

\begin{itemize}
\item Taking $p=k$ and $q=1$ in the Eqs. (3.21)-(3.24),\ we get the new
generating functions of even and odd $k$-Jacobsthal and $k$-Jacobsthal Lucas
numbers. The calculation and results are listed in the Tab. 4.%
\begin{equation*}
\begin{tabular}{|c|c|}
\hline
C$\text{oefficient of }z^{n}$ & Generating function \\ \hline
$J_{k,2n}$ & $\frac{kz}{1-\left( k^{2}+4\right) z+4z^{2}}$ \\ \hline
$J_{k,2n+1}$ & $\frac{1-2z}{1-\left( k^{2}+4\right) z+4z^{2}}$ \\ \hline
$j_{k,2n}$ & $\frac{2-\left( k^{2}+4\right) z}{1-\left( k^{2}+4\right)
z+4z^{2}}$ \\ \hline
$j_{k,2n+1}$ & $\frac{k+2kz}{1-\left( k^{2}+4\right) z+4z^{2}}$ \\ \hline
\end{tabular}%
\end{equation*}
\end{itemize}

\begin{center}
\textbf{Table 4.} New generating functions for even and odd $k$-Jacobsthal
and $k$-Jacobsthal Lucas numbers.
\end{center}

\begin{remark}
If we take $k=1$ in the Tab. 4,\ we get the generating functions of even and
odd Jacobsthal and Jacobsthal Lucas numbers given in the article of Mezo 
\cite{MEZO} as follows:%
\begin{equation*}
\sum\limits_{n=0}^{\infty }J_{2n}z^{n}=\frac{z}{1-5z+4z^{2}},\text{ }%
\sum\limits_{n=0}^{\infty }J_{2n+1}z^{n}=\frac{1-2z}{1-5z+4z^{2}},\text{ }%
\sum\limits_{n=0}^{\infty }j_{2n}z^{n}=\frac{2-5z}{1-5z+4z^{2}}\text{ and }%
\sum\limits_{n=0}^{\infty }j_{2n+1}z^{n}=\frac{1+2z}{1-5z+4z^{2}}.
\end{equation*}
\end{remark}

\section{\textbf{Generating functions of the products of }$\left( 
\boldsymbol{p,q}\right) $\textbf{-numbers with odd and even }$\left( 
\boldsymbol{p,q}\right) $\textbf{-numbers}}

We now consider the Theorems 1, 2, 3, 4 and 5 in order to derive a new
generating functions of binary products of $\left( p,q\right) $-numbers with
odd and even $\left( p,q\right) $-numbers.

We consider the following sets:%
\begin{equation*}
A=\left\{ a_{1},-a_{2}\right\} \text{ and }E=\left\{ e_{1},-e_{2}\right\} 
\text{.}
\end{equation*}%
By changing $a_{2}$ to $\left( -a_{2}\right) $ and $e_{2}$ to $\left(
-e_{2}\right) $ in Eqs. (2.1), (2.2), (2.3), (2.4), (2.6) and (2.8), it
becomes%
\begin{equation}
\sum\limits_{n=0}^{\infty }S_{n}\left( a_{1}+\left[ -a_{2}\right] \right)
S_{2n-1}\left( e_{1}+\left[ -e_{2}\right] \right) z^{n}=\frac{\left(
a_{1}-a_{2}\right) \left( e_{1}-e_{2}\right) z+a_{1}a_{2}\left(
e_{1}-e_{2}\right) \left( \left( e_{1}-e_{2}\right) ^{2}+2e_{1}e_{2}\right)
z^{2}}{P(z)},  \tag{4.1}
\end{equation}%
\begin{equation}
\sum\limits_{n=0}^{\infty }S_{n}\left( a_{1}+\left[ -a_{2}\right] \right)
S_{2n}\left( e_{1}+\left[ -e_{2}\right] \right) z^{n}=\frac{%
1-e_{1}e_{2}\left( a_{1}-a_{2}\right) z-e_{1}e_{2}a_{1}a_{2}\left( \left(
e_{1}-e_{2}\right) ^{2}+e_{1}e_{2}\right) z^{2}}{P(z)},  \tag{4.2}
\end{equation}%
\begin{equation}
\sum\limits_{n=0}^{\infty }S_{n}\left( a_{1}+\left[ -a_{2}\right] \right)
S_{2n+1}\left( e_{1}+\left[ -e_{2}\right] \right) z^{n}=\frac{%
e_{1}-e_{2}+a_{1}a_{2}e_{1}^{2}e_{2}^{2}\left( e_{1}-e_{2}\right) z^{2}}{P(z)%
},  \tag{4.3}
\end{equation}%
\begin{equation}
\sum\limits_{n=0}^{\infty }S_{n-1}\left( a_{1}+\left[ -a_{2}\right] \right)
S_{2n-1}\left( e_{1}+\left[ -e_{2}\right] \right) z^{n}=\frac{\left(
e_{1}-e_{2}\right) z+a_{1}a_{2}e_{1}^{2}e_{2}^{2}\left( e_{1}-e_{2}\right)
z^{3}}{P(z)},  \tag{4.4}
\end{equation}%
\begin{equation}
\sum\limits_{n=0}^{\infty }S_{n-1}\left( a_{1}+\left[ -a_{2}\right] \right)
S_{2n}\left( e_{1}+\left[ -e_{2}\right] \right) z^{n}=\frac{\left( \left(
e_{1}-e_{2}\right) ^{2}+e_{1}e_{2}\right) z-e_{1}^{2}e_{2}^{2}\left(
a_{1}-a_{2}\right) z^{2}-a_{1}a_{2}e_{1}^{3}e_{2}^{3}z^{3}}{P(z)},  \tag{4.5}
\end{equation}%
\begin{equation}
\sum\limits_{n=0}^{\infty }S_{n-1}\left( a_{1}+\left[ -a_{2}\right] \right)
S_{2n+1}\left( e_{1}+\left[ -e_{2}\right] \right) z^{n}=\frac{\left(
e_{1}-e_{2}\right) \left( \left( e_{1}-e_{2}\right) ^{2}+2e_{1}e_{2}\right)
z-e_{1}^{2}e_{2}^{2}\left( a_{1}-a_{2}\right) \left( e_{1}-e_{2}\right) z^{2}%
}{P(z)},  \tag{4.6}
\end{equation}%
respectively, with $%
P(z)=(1-a_{1}e_{1}^{2}z)(1-a_{1}e_{2}^{2}z)(1+a_{2}e_{1}^{2}z)(1+a_{2}e_{2}^{2}z),\ 
$and we have three cases.

\textbf{Case 1. }Let us now consider the following conditions for Eqs.
(4.1)-(4.6):%
\begin{equation*}
\left\{ 
\begin{array}{l}
a_{1}-a_{2}=p \\ 
a_{1}a_{2}=q%
\end{array}%
\right. \ \text{and }\left\{ 
\begin{array}{l}
e_{1}-e_{2}=p \\ 
e_{1}e_{2}=q%
\end{array}%
\right. .
\end{equation*}

Then it yields%
\begin{equation}
\sum\limits_{n=0}^{\infty }S_{n}\left( a_{1}+\left[ -a_{2}\right] \right)
S_{2n-1}\left( e_{1}+\left[ -e_{2}\right] \right) z^{n}=\frac{%
p^{2}z+pq\left( p^{2}+2q\right) z^{2}}{D_{1}},  \tag{4.7}
\end{equation}%
\begin{equation}
\sum\limits_{n=0}^{\infty }S_{n}\left( a_{1}+\left[ -a_{2}\right] \right)
S_{2n}\left( e_{1}+\left[ -e_{2}\right] \right) z^{n}=\frac{%
1-pqz-q^{2}\left( p^{2}+q\right) z^{2}}{D_{1}},  \tag{4.8}
\end{equation}%
\begin{equation}
\sum\limits_{n=0}^{\infty }S_{n}\left( a_{1}+\left[ -a_{2}\right] \right)
S_{2n+1}\left( e_{1}+\left[ -e_{2}\right] \right) z^{n}=\frac{p+pq^{3}z^{2}}{%
D_{1}},  \tag{4.9}
\end{equation}%
\begin{equation}
\sum\limits_{n=0}^{\infty }S_{n-1}\left( a_{1}+\left[ -a_{2}\right] \right)
S_{2n-1}\left( e_{1}+\left[ -e_{2}\right] \right) z^{n}=\frac{pz+pq^{3}z^{3}%
}{D_{1}},  \tag{4.10}
\end{equation}%
\begin{equation}
\sum\limits_{n=0}^{\infty }S_{n-1}\left( a_{1}+\left[ -a_{2}\right] \right)
S_{2n}\left( e_{1}+\left[ -e_{2}\right] \right) z^{n}=\frac{\left(
p^{2}+q\right) z-pq^{2}z^{2}-q^{4}z^{3}}{D_{1}},  \tag{4.11}
\end{equation}%
\begin{equation}
\sum\limits_{n=0}^{\infty }S_{n-1}\left( a_{1}+\left[ -a_{2}\right] \right)
S_{2n+1}\left( e_{1}+\left[ -e_{2}\right] \right) z^{n}=\frac{p\left(
p^{2}+2q\right) z-p^{2}q^{2}z^{2}}{D_{1}},  \tag{4.12}
\end{equation}%
with 
\begin{equation*}
D_{1}=1-p\left( p^{2}+2q\right) z-q\left( p^{4}+3p^{2}q+2q^{2}\right)
z^{2}+pq^{3}\left( p^{2}+2q\right) z^{3}+q^{6}z^{4}.
\end{equation*}%
And we deduce the following corollaries and theorems.

\begin{corollary}
Let $n$ be a natural number. Then we have the new generating function of the
product of $\left( p,q\right) $-Fibonacci numbers $\left(
F_{p,q,n}F_{p,q,2n}\right) $:%
\begin{eqnarray}
\dsum\limits_{n=0}^{\infty }F_{p,q,n}F_{p,q,2n}z^{n} &=&\frac{pz+pq^{3}z^{3}%
}{D_{1}}  \notag \\
&=&\frac{pz+pq^{3}z^{3}}{1-p\left( p^{2}+2q\right) z-q\left(
p^{4}+3p^{2}q+2q^{2}\right) z^{2}+pq^{3}\left( p^{2}+2q\right)
z^{3}+q^{6}z^{4}},  \TCItag{4.13}
\end{eqnarray}%
with $F_{p,q,n}F_{p,q,2n}=S_{n-1}(a_{1}+[-a_{2}])S_{2n-1}(e_{1}+[-e_{2}]).$
\end{corollary}

\begin{corollary}
For $n\in 
\mathbb{N}
,$ the new generating function of the product of $\left( p,q\right) $%
-Fibonacci numbers $\left( F_{p,q,n}F_{p,q,2n+1}\right) $ is given by:%
\begin{eqnarray}
\dsum\limits_{n=0}^{\infty }F_{p,q,n}F_{p,q,2n+1}z^{n} &=&\frac{\left(
p^{2}+q\right) z-pq^{2}z^{2}-q^{4}z^{3}}{D_{1}}  \notag \\
&=&\frac{\left( p^{2}+q\right) z-pq^{2}z^{2}-q^{4}z^{3}}{1-p\left(
p^{2}+2q\right) z-q\left( p^{4}+3p^{2}q+2q^{2}\right) z^{2}+pq^{3}\left(
p^{2}+2q\right) z^{3}+q^{6}z^{4}},  \TCItag{4.14}
\end{eqnarray}%
with $F_{p,q,n}F_{p,q,2n+1}=S_{n-1}(a_{1}+[-a_{2}])S_{2n}(e_{1}+[-e_{2}]).$
\end{corollary}

\begin{theorem}
For $n\in 
\mathbb{N}
,$ the new generating function of the product of $\left( p,q\right) $-Lucas
numbers $\left( L_{p,q,n}L_{p,q,2n}\right) $ is given by:%
\begin{equation}
\dsum\limits_{n=0}^{\infty }L_{p,q,n}L_{p,q,2n}z^{n}=\frac{4-3p\left(
p^{2}+2q\right) z-2q\left( p^{4}+3p^{2}q+2q^{2}\right) z^{2}+pq^{3}\left(
p^{2}+2q\right) z^{3}}{1-p\left( p^{2}+2q\right) z-q\left(
p^{4}+3p^{2}q+2q^{2}\right) z^{2}+pq^{3}\left( p^{2}+2q\right)
z^{3}+q^{6}z^{4}}.  \tag{4.15}
\end{equation}
\end{theorem}

\begin{proof}
We have%
\begin{eqnarray*}
\dsum\limits_{n=0}^{\infty }L_{p,q,n}L_{p,q,2n}z^{n}
&=&\dsum\limits_{n=0}^{\infty }\left( 
\begin{array}{c}
\left( 2S_{n}(a_{1}+[-a_{2}])-pS_{n-1}(a_{1}+[-a_{2}])\right) \\ 
\times \left( 2S_{2n}(e_{1}+[-e_{2}])-pS_{2n-1}(e_{1}+[-e_{2}])\right)%
\end{array}%
\right) z^{n} \\
&=&4\dsum\limits_{n=0}^{\infty
}S_{n}(a_{1}+[-a_{2}])S_{2n}(e_{1}+[-e_{2}])z^{n} \\
&&-2p\dsum\limits_{n=0}^{\infty
}S_{n}(a_{1}+[-a_{2}])S_{2n-1}(e_{1}+[-e_{2}])z^{n} \\
&&-2p\dsum\limits_{n=0}^{\infty
}S_{n-1}(a_{1}+[-a_{2}])S_{2n}(e_{1}+[-e_{2}])z^{n} \\
&&+p^{2}\dsum\limits_{n=0}^{\infty
}S_{n-1}(a_{1}+[-a_{2}])S_{2n-1}(e_{1}+[-e_{2}])z^{n}.
\end{eqnarray*}

Using the relationships $\left( 4.7\right) ,$ $\left( 4.8\right) ,$ $\left(
4.10\right) $ and $\left( 4.11\right) $, we obtain%
\begin{eqnarray*}
\dsum\limits_{n=0}^{\infty }L_{p,q,n}L_{p,q,2n}z^{n} &=&\frac{4\left(
1-pqz-q^{2}\left( p^{2}+q\right) z^{2}\right) }{D_{1}}-\frac{2p\left(
p^{2}z+pq\left( p^{2}+2q\right) z^{2}\right) }{D_{1}} \\
&&-\frac{2p\left( \left( p^{2}+q\right) z-pq^{2}z^{2}-q^{4}z^{3}\right) }{%
D_{1}}+\frac{p^{2}\left( pz+pq^{3}z^{3}\right) }{D_{1}} \\
&=&\frac{4-3p\left( p^{2}+2q\right) z-2q\left( p^{4}+3p^{2}q+2q^{2}\right)
z^{2}+pq^{3}\left( p^{2}+2q\right) z^{3}}{D_{1}} \\
&=&\frac{4-3p\left( p^{2}+2q\right) z-2q\left( p^{4}+3p^{2}q+2q^{2}\right)
z^{2}+pq^{3}\left( p^{2}+2q\right) z^{3}}{1-p\left( p^{2}+2q\right)
z-q\left( p^{4}+3p^{2}q+2q^{2}\right) z^{2}+pq^{3}\left( p^{2}+2q\right)
z^{3}+q^{6}z^{4}}.
\end{eqnarray*}

So, the proof is completed.
\end{proof}

\begin{theorem}
Let $n$ be a natural number. Then we have the new generating function of the
product of $\left( p,q\right) $-Lucas numbers $\left(
L_{p,q,n}L_{p,q,2n+1}\right) $:%
\begin{equation}
\dsum\limits_{n=0}^{\infty }L_{p,q,n}L_{p,q,2n+1}z^{n}=\frac{2p-p^{2}\left(
p^{2}+q\right) z+3pq^{2}\left( p^{2}+2q\right) z^{2}-p^{2}q^{4}z^{3}}{%
1-p\left( p^{2}+2q\right) z-q\left( p^{4}+3p^{2}q+2q^{2}\right)
z^{2}+pq^{3}\left( p^{2}+2q\right) z^{3}+q^{6}z^{4}}.  \tag{4.16}
\end{equation}
\end{theorem}

\begin{proof}
We have%
\begin{eqnarray*}
\dsum\limits_{n=0}^{\infty }L_{p,q,n}L_{p,q,2n+1}z^{n}
&=&\dsum\limits_{n=0}^{\infty }\left( 
\begin{array}{c}
\left( 2S_{n}(a_{1}+[-a_{2}])-pS_{n-1}(a_{1}+[-a_{2}])\right) \\ 
\times \left( 2S_{2n+1}(e_{1}+[-e_{2}])-pS_{2n}(e_{1}+[-e_{2}])\right)%
\end{array}%
\right) z^{n} \\
&=&4\dsum\limits_{n=0}^{\infty
}S_{n}(a_{1}+[-a_{2}])S_{2n+1}(e_{1}+[-e_{2}])z^{n} \\
&&-2p\dsum\limits_{n=0}^{\infty
}S_{n}(a_{1}+[-a_{2}])S_{2n}(e_{1}+[-e_{2}])z^{n} \\
&&-2p\dsum\limits_{n=0}^{\infty
}S_{n-1}(a_{1}+[-a_{2}])S_{2n+1}(e_{1}+[-e_{2}])z^{n} \\
&&+p^{2}\dsum\limits_{n=0}^{\infty
}S_{n-1}(a_{1}+[-a_{2}])S_{2n}(e_{1}+[-e_{2}])z^{n}.
\end{eqnarray*}

Using the relationships $\left( 4.8\right) ,$ $\left( 4.9\right) ,$ $\left(
4.11\right) $ and $\left( 4.12\right) $, we obtain%
\begin{eqnarray*}
\dsum\limits_{n=0}^{\infty }L_{p,q,n}L_{p,q,2n+1}z^{n} &=&\frac{4\left(
p+pq^{3}z^{2}\right) }{D_{1}}-\frac{2p\left( 1-pqz-q^{2}\left(
p^{2}+q\right) z^{2}\right) }{D_{1}} \\
&&-\frac{2p\left( p\left( p^{2}+2q\right) z-p^{2}q^{2}z^{2}\right) }{D_{1}}+%
\frac{p^{2}\left( \left( p^{2}+q\right) z-pq^{2}z^{2}-q^{4}z^{3}\right) }{%
D_{1}} \\
&=&\frac{2p-p^{2}\left( p^{2}+q\right) z+3pq^{2}\left( p^{2}+2q\right)
z^{2}-p^{2}q^{4}z^{3}}{D_{1}} \\
&=&\frac{2p-p^{2}\left( p^{2}+q\right) z+3pq^{2}\left( p^{2}+2q\right)
z^{2}-p^{2}q^{4}z^{3}}{1-p\left( p^{2}+2q\right) z-q\left(
p^{4}+3p^{2}q+2q^{2}\right) z^{2}+pq^{3}\left( p^{2}+2q\right)
z^{3}+q^{6}z^{4}}.
\end{eqnarray*}

So, the proof is completed.
\end{proof}

\begin{corollary}
Putting $p=k$ and $q=1$ in Eqs. (4.13)-(4.16) gives the following new
generating functions:%
\begin{eqnarray*}
\dsum\limits_{n=0}^{\infty }F_{k,n}F_{k,2n}z^{n} &=&\frac{kz+kz^{3}}{%
1-k\left( k^{2}+2\right) z-\left( k^{4}+3k^{2}+2\right) z^{2}+k\left(
k^{2}+2\right) z^{3}+z^{4}}. \\
\dsum\limits_{n=0}^{\infty }F_{k,n}F_{k,2n+1}z^{n} &=&\frac{\left(
k^{2}+1\right) z-kz^{2}-z^{3}}{1-k\left( k^{2}+2\right) z-\left(
k^{4}+3k^{2}+2\right) z^{2}+k\left( k^{2}+2\right) z^{3}+z^{4}}. \\
\dsum\limits_{n=0}^{\infty }L_{k,n}L_{k,2n}z^{n} &=&\frac{4-3k\left(
k^{2}+2\right) z-2\left( k^{4}+3k^{2}+2\right) z^{2}+k\left( k^{2}+2\right)
z^{3}}{1-k\left( k^{2}+2\right) z-\left( k^{4}+3k^{2}+2\right) z^{2}+k\left(
k^{2}+2\right) z^{3}+z^{4}}. \\
\dsum\limits_{n=0}^{\infty }L_{k,n}L_{k,2n+1}z^{n} &=&\frac{2k-k^{2}\left(
k^{2}+1\right) z+3k\left( k^{2}+2\right) z^{2}-k^{2}z^{3}}{1-k\left(
k^{2}+2\right) z-\left( k^{4}+3k^{2}+2\right) z^{2}+k\left( k^{2}+2\right)
z^{3}+z^{4}}.
\end{eqnarray*}
\end{corollary}

\begin{itemize}
\item Put $k=1$ in the Corollary 3, we obtain the following table:%
\begin{equation*}
\begin{tabular}{|c|c|}
\hline
$\text{Coefficient of }z^{n}$ & Generating function \\ \hline
$F_{n}F_{2n}$ & $\frac{z+z^{3}}{1-3z-6z^{2}+3z^{3}+z^{4}}$ \\ \hline
$F_{n}F_{2n+1}$ & $\frac{2z-z^{2}-z^{3}}{1-3z-6z^{2}+3z^{3}+z^{4}}$ \\ \hline
$L_{n}L_{2n}$ & $\frac{4-9z-12z^{2}+3z^{3}}{1-3z-6z^{2}+3z^{3}+z^{4}}$ \\ 
\hline
$L_{n}L_{2n+1}$ & $\frac{2-2z+9z^{2}-z^{3}}{1-3z-6z^{2}+3z^{3}+z^{4}}$ \\ 
\hline
\end{tabular}%
\end{equation*}
\end{itemize}

\begin{center}
\textbf{Table 5.} A new generating functions of the products of some numbers.
\end{center}

\textbf{Case 2. }Let us now consider the following conditions for Eqs.
(4.1)-(4.6):%
\begin{equation*}
\left\{ 
\begin{array}{l}
a_{1}-a_{2}=2p \\ 
a_{1}a_{2}=q%
\end{array}%
\right. \text{ and }\left\{ 
\begin{array}{l}
e_{1}-e_{2}=2p \\ 
e_{1}e_{2}=q%
\end{array}%
\right. .
\end{equation*}

Then it yields%
\begin{equation}
\sum\limits_{n=0}^{\infty }S_{n}\left( a_{1}+\left[ -a_{2}\right] \right)
S_{2n-1}\left( e_{1}+\left[ -e_{2}\right] \right) z^{n}=\frac{%
4p^{2}z+4pq\left( 2p^{2}+q\right) z^{2}}{D_{2}},  \tag{4.17}
\end{equation}%
\begin{equation}
\sum\limits_{n=0}^{\infty }S_{n}\left( a_{1}+\left[ -a_{2}\right] \right)
S_{2n}\left( e_{1}+\left[ -e_{2}\right] \right) z^{n}=\frac{%
1-2pqz-q^{2}\left( 4p^{2}+q\right) z^{2}}{D_{2}},  \tag{4.18}
\end{equation}%
\begin{equation}
\sum\limits_{n=0}^{\infty }S_{n}\left( a_{1}+\left[ -a_{2}\right] \right)
S_{2n+1}\left( e_{1}+\left[ -e_{2}\right] \right) z^{n}=\frac{2p+2pq^{3}z^{2}%
}{D_{2}},  \tag{4.19}
\end{equation}%
\begin{equation}
\sum\limits_{n=0}^{\infty }S_{n-1}\left( a_{1}+\left[ -a_{2}\right] \right)
S_{2n-1}\left( e_{1}+\left[ -e_{2}\right] \right) z^{n}=\frac{%
2pz+2pq^{3}z^{3}}{D_{2}},  \tag{4.20}
\end{equation}%
\begin{equation}
\sum\limits_{n=0}^{\infty }S_{n-1}\left( a_{1}+\left[ -a_{2}\right] \right)
S_{2n}\left( e_{1}+\left[ -e_{2}\right] \right) z^{n}=\frac{\left(
4p^{2}+q\right) z-2pq^{2}z^{2}-q^{4}z^{3}}{D_{2}},  \tag{4.21}
\end{equation}%
\begin{equation}
\sum\limits_{n=0}^{\infty }S_{n-1}\left( a_{1}+\left[ -a_{2}\right] \right)
S_{2n+1}\left( e_{1}+\left[ -e_{2}\right] \right) z^{n}=\frac{4p\left(
2p^{2}+q\right) z-4p^{2}q^{2}z^{2}}{D_{2}},  \tag{4.22}
\end{equation}%
with%
\begin{equation*}
D_{2}=1-4p\left( 2p^{2}+q\right) z-2q\left( 8p^{4}+6p^{2}q+q^{2}\right)
z^{2}+4pq^{3}\left( 2p^{2}+q\right) z^{3}+q^{6}z^{4}.
\end{equation*}%
And we deduce the following corollaries and theorems.

\begin{corollary}
For $n\in 
\mathbb{N}
,$ the new generating function of the product of $\left( p,q\right) $-Pell
numbers $\left( P_{p,q,n}P_{p,q,2n}\right) $ is given by:%
\begin{eqnarray}
\dsum\limits_{n=0}^{\infty }P_{p,q,n}P_{p,q,2n}z^{n} &=&\frac{%
2pz+2pq^{3}z^{3}}{D_{2}}  \notag \\
&=&\frac{2pz+2pq^{3}z^{3}}{1-4p\left( 2p^{2}+q\right) z-2q\left(
8p^{4}+6p^{2}q+q^{2}\right) z^{2}+4pq^{3}\left( 2p^{2}+q\right)
z^{3}+q^{6}z^{4}},  \TCItag{4.23}
\end{eqnarray}%
with $P_{p,q,n}P_{p,q,2n}=S_{n-1}(a_{1}+[-a_{2}])S_{2n-1}(e_{1}+[-e_{2}]).$
\end{corollary}

\begin{corollary}
Let $n$ be a natural number. Then we have the new generating function of the
product of $\left( p,q\right) $-Pell numbers $\left(
P_{p,q,n}P_{p,q,2n+1}\right) :$%
\begin{eqnarray}
\dsum\limits_{n=0}^{\infty }P_{p,q,n}P_{p,q,2n+1}z^{n} &=&\frac{\left(
4p^{2}+q\right) z-2pq^{2}z^{2}-q^{4}z^{3}}{D_{2}}  \notag \\
&=&\frac{\left( 4p^{2}+q\right) z-2pq^{2}z^{2}-q^{4}z^{3}}{1-4p\left(
2p^{2}+q\right) z-2q\left( 8p^{4}+6p^{2}q+q^{2}\right) z^{2}+4pq^{3}\left(
2p^{2}+q\right) z^{3}+q^{6}z^{4}},  \TCItag{4.24}
\end{eqnarray}%
with $P_{p,q,n}P_{p,q,2n+1}=S_{n-1}(a_{1}+[-a_{2}])S_{2n}(e_{1}+[-e_{2}]).$
\end{corollary}

\begin{theorem}
For $n\in 
\mathbb{N}
,$ the new generating function of the product of $\left( p,q\right) $-Pell
Lucas numbers $\left( Q_{p,q,n}Q_{p,q,2n}\right) $ is given by:%
\begin{equation}
\dsum\limits_{n=0}^{\infty }Q_{p,q,n}Q_{p,q,2n}z^{n}=\frac{4-12p\left(
2p^{2}+q\right) z-4q\left( 8p^{4}+4p^{2}q+q^{2}\right) z^{2}+4pq^{3}\left(
2p^{2}+q\right) z^{3}}{1-4p\left( 2p^{2}+q\right) z-2q\left(
8p^{4}+6p^{2}q+q^{2}\right) z^{2}+4pq^{3}\left( 2p^{2}+q\right)
z^{3}+q^{6}z^{4}}.  \tag{4.25}
\end{equation}
\end{theorem}

\begin{proof}
We have%
\begin{eqnarray*}
\dsum\limits_{n=0}^{\infty }Q_{p,q,n}Q_{p,q,2n}z^{n}
&=&\dsum\limits_{n=0}^{\infty }\left( 
\begin{array}{c}
\left( 2S_{n}(a_{1}+[-a_{2}])-2pS_{n-1}(a_{1}+[-a_{2}])\right) \\ 
\times \left( 2S_{2n}(e_{1}+[-e_{2}])-2pS_{2n-1}(e_{1}+[-e_{2}])\right)%
\end{array}%
\right) z^{n} \\
&=&4\dsum\limits_{n=0}^{\infty
}S_{n}(a_{1}+[-a_{2}])S_{2n}(e_{1}+[-e_{2}])z^{n} \\
&&-4p\dsum\limits_{n=0}^{\infty
}S_{n}(a_{1}+[-a_{2}])S_{2n-1}(e_{1}+[-e_{2}])z^{n} \\
&&-4p\dsum\limits_{n=0}^{\infty
}S_{n-1}(a_{1}+[-a_{2}])S_{2n}(e_{1}+[-e_{2}])z^{n} \\
&&+4p^{2}\dsum\limits_{n=0}^{\infty
}S_{n-1}(a_{1}+[-a_{2}])S_{2n-1}(e_{1}+[-e_{2}])z^{n}.
\end{eqnarray*}

Using the relationships $\left( 4.17\right) ,$ $\left( 4.18\right) ,$ $%
\left( 4.20\right) $ and $\left( 4.21\right) $, we obtain%
\begin{eqnarray*}
\dsum\limits_{n=0}^{\infty }Q_{p,q,n}Q_{p,q,2n}z^{n} &=&\frac{4\left(
1-2pqz-q^{2}\left( 4p^{2}+q\right) z^{2}\right) }{D_{2}}-\frac{4p\left(
4p^{2}z+4pq\left( 2p^{2}+q\right) z^{2}\right) }{D_{2}} \\
&&-\frac{4p\left( \left( 4p^{2}+q\right) z-2pq^{2}z^{2}-q^{4}z^{3}\right) }{%
D_{2}}+\frac{4p^{2}\left( 2pz+2pq^{3}z^{3}\right) }{D_{2}} \\
&=&\frac{4-12p\left( 2p^{2}+q\right) z-4q\left( 8p^{4}+4p^{2}q+q^{2}\right)
z^{2}+4pq^{3}\left( 2p^{2}+q\right) z^{3}}{D_{2}} \\
&=&\frac{4-12p\left( 2p^{2}+q\right) z-4q\left( 8p^{4}+4p^{2}q+q^{2}\right)
z^{2}+4pq^{3}\left( 2p^{2}+q\right) z^{3}}{1-4p\left( 2p^{2}+q\right)
z-2q\left( 8p^{4}+6p^{2}q+q^{2}\right) z^{2}+4pq^{3}\left( 2p^{2}+q\right)
z^{3}+q^{6}z^{4}}.
\end{eqnarray*}

So, the proof is completed.
\end{proof}

\begin{theorem}
Let $n$ be a natural number. Then we have the new generating function of the
product of $\left( p,q\right) $-Pell Lucas numbers $\left(
Q_{p,q,n}Q_{p,q,2n+1}\right) $:%
\begin{equation}
\dsum\limits_{n=0}^{\infty }Q_{p,q,n}Q_{p,q,2n+1}z^{n}=\frac{4p-4p^{2}\left(
4p^{2}+q\right) z+12pq^{2}\left( 2p^{2}+q\right) z^{2}-4p^{2}q^{4}z^{3}}{%
1-4p\left( 2p^{2}+q\right) z-2q\left( 8p^{4}+6p^{2}q+q^{2}\right)
z^{2}+4pq^{3}\left( 2p^{2}+q\right) z^{3}+q^{6}z^{4}}.  \tag{4.26}
\end{equation}
\end{theorem}

\begin{proof}
We have%
\begin{eqnarray*}
\dsum\limits_{n=0}^{\infty }Q_{p,q,n}Q_{p,q,2n+1}z^{n}
&=&\dsum\limits_{n=0}^{\infty }\left( 
\begin{array}{c}
\left( 2S_{n}(a_{1}+[-a_{2}])-2pS_{n-1}(a_{1}+[-a_{2}])\right) \\ 
\times \left( 2S_{2n+1}(e_{1}+[-e_{2}])-2pS_{2n}(e_{1}+[-e_{2}])\right)%
\end{array}%
\right) z^{n} \\
&=&4\dsum\limits_{n=0}^{\infty
}S_{n}(a_{1}+[-a_{2}])S_{2n+1}(e_{1}+[-e_{2}])z^{n} \\
&&-4p\dsum\limits_{n=0}^{\infty
}S_{n}(a_{1}+[-a_{2}])S_{2n}(e_{1}+[-e_{2}])z^{n} \\
&&-4p\dsum\limits_{n=0}^{\infty
}S_{n-1}(a_{1}+[-a_{2}])S_{2n+1}(e_{1}+[-e_{2}])z^{n} \\
&&+4p^{2}\dsum\limits_{n=0}^{\infty
}S_{n-1}(a_{1}+[-a_{2}])S_{2n}(e_{1}+[-e_{2}])z^{n}.
\end{eqnarray*}

Using the relationships $\left( 4.18\right) ,$ $\left( 4.19\right) ,$ $%
\left( 4.21\right) $ and $\left( 4.22\right) $, we obtain%
\begin{eqnarray*}
\dsum\limits_{n=0}^{\infty }Q_{p,q,n}Q_{p,q,2n+1}z^{n} &=&\frac{4\left(
2p+2pq^{3}z^{2}\right) }{D_{2}}-\frac{4p\left( 1-2pqz-q^{2}\left(
4p^{2}+q\right) z^{2}\right) }{D_{2}} \\
&&-\frac{4p\left( 4p\left( 2p^{2}+q\right) z-4p^{2}q^{2}z^{2}\right) }{D_{2}}%
+\frac{4p^{2}\left( \left( 4p^{2}+q\right) z-2pq^{2}z^{2}-q^{4}z^{3}\right) 
}{D_{2}} \\
&=&\frac{4p-4p^{2}\left( 4p^{2}+q\right) z+12pq^{2}\left( 2p^{2}+q\right)
z^{2}-4p^{2}q^{4}z^{3}}{D_{2}} \\
&=&\frac{4p-4p^{2}\left( 4p^{2}+q\right) z+12pq^{2}\left( 2p^{2}+q\right)
z^{2}-4p^{2}q^{4}z^{3}}{1-4p\left( 2p^{2}+q\right) z-2q\left(
8p^{4}+6p^{2}q+q^{2}\right) z^{2}+4pq^{3}\left( 2p^{2}+q\right)
z^{3}+q^{6}z^{4}}.
\end{eqnarray*}

So, the proof is completed.
\end{proof}

\begin{corollary}
Taking $p=1$ and $q=k$ in Eqs. (4.23)-(4.26) gives the following new
generating functions:%
\begin{eqnarray*}
\dsum\limits_{n=0}^{\infty }P_{k,n}P_{k,2n}z^{n} &=&\frac{2z+2k^{3}z^{3}}{%
1-4\left( k+2\right) z-2k\left( k^{2}+6k+8\right) z^{2}+4k^{3}\left(
k+2\right) z^{3}+k^{6}z^{4}}. \\
\dsum\limits_{n=0}^{\infty }P_{k,n}P_{k,2n+1}z^{n} &=&\frac{\left(
k+4\right) z-2k^{2}z^{2}-k^{4}z^{3}}{1-4\left( k+2\right) z-2k\left(
k^{2}+6k+8\right) z^{2}+4k^{3}\left( k+2\right) z^{3}+k^{6}z^{4}}. \\
\dsum\limits_{n=0}^{\infty }Q_{k,n}Q_{k,2n}z^{n} &=&\frac{4-12\left(
k+2\right) z-4k\left( k^{2}+4k+8\right) z^{2}+4k^{3}\left( k+2\right) z^{3}}{%
1-4\left( k+2\right) z-2k\left( k^{2}+6k+8\right) z^{2}+4k^{3}\left(
k+2\right) z^{3}+k^{6}z^{4}}. \\
\dsum\limits_{n=0}^{\infty }Q_{k,n}Q_{k,2n+1}z^{n} &=&\frac{4-4\left(
k+4\right) z+12k^{2}\left( k+2\right) z^{2}-4k^{4}z^{3}}{1-4\left(
k+2\right) z-2k\left( k^{2}+6k+8\right) z^{2}+4k^{3}\left( k+2\right)
z^{3}+k^{6}z^{4}}.
\end{eqnarray*}
\end{corollary}

\begin{itemize}
\item Put $k=1$ in the Corollary 6, we obtain the following table:%
\begin{equation*}
\begin{tabular}{|c|c|}
\hline
$\text{Coefficient of }z^{n}$ & Generating function \\ \hline
$P_{n}P_{2n}$ & $\frac{2z+2z^{3}}{1-12z-30z^{2}+12z^{3}+z^{4}}$ \\ \hline
$P_{n}P_{2n+1}$ & $\frac{5z-2z^{2}-z^{3}}{1-12z-30z^{2}+12z^{3}+z^{4}}$ \\ 
\hline
$Q_{n}Q_{2n}$ & $\frac{4-36z-52z^{2}+12z^{3}}{1-12z-30z^{2}+12z^{3}+z^{4}}$
\\ \hline
$Q_{n}Q_{2n+1}$ & $\frac{4-20z+36z^{2}-4z^{3}}{1-12z-30z^{2}+12z^{3}+z^{4}}$
\\ \hline
\end{tabular}%
\end{equation*}
\end{itemize}

\begin{center}
\textbf{Table 6.} A new generating functions of the products of some numbers.
\end{center}

\textbf{Case 3. }Let us now consider the following conditions for Eqs.
(4.1)-(4.6)%
\begin{equation*}
\left\{ 
\begin{array}{c}
a_{1}-a_{2}=p \\ 
a_{1}a_{2}=2q%
\end{array}%
\right. \text{ and }\left\{ 
\begin{array}{c}
e_{1}-e_{2}=p \\ 
e_{1}e_{2}=2q%
\end{array}%
\right. .
\end{equation*}

Then it give%
\begin{equation}
\sum\limits_{n=0}^{\infty }S_{n}\left( a_{1}+\left[ -a_{2}\right] \right)
S_{2n-1}\left( e_{1}+\left[ -e_{2}\right] \right) z^{n}=\frac{%
p^{2}z+2pq\left( p^{2}+4q\right) z^{2}}{D_{3}},  \tag{4.27}
\end{equation}%
\begin{equation}
\sum\limits_{n=0}^{\infty }S_{n}\left( a_{1}+\left[ -a_{2}\right] \right)
S_{2n}\left( e_{1}+\left[ -e_{2}\right] \right) z^{n}=\frac{%
1-2pqz-4q^{2}\left( p^{2}+2q\right) z^{2}}{D_{3}},  \tag{4.28}
\end{equation}%
\begin{equation}
\sum\limits_{n=0}^{\infty }S_{n}\left( a_{1}+\left[ -a_{2}\right] \right)
S_{2n+1}\left( e_{1}+\left[ -e_{2}\right] \right) z^{n}=\frac{p+8pq^{3}z^{2}%
}{D_{3}},  \tag{4.29}
\end{equation}%
\begin{equation}
\sum\limits_{n=0}^{\infty }S_{n-1}\left( a_{1}+\left[ -a_{2}\right] \right)
S_{2n-1}\left( e_{1}+\left[ -e_{2}\right] \right) z^{n}=\frac{pz+8pq^{3}z^{3}%
}{D_{3}},  \tag{4.30}
\end{equation}%
\begin{equation}
\sum\limits_{n=0}^{\infty }S_{n-1}\left( a_{1}+\left[ -a_{2}\right] \right)
S_{2n}\left( e_{1}+\left[ -e_{2}\right] \right) z^{n}=\frac{\left(
p^{2}+2q\right) z-4pq^{2}z^{2}-16q^{4}z^{3}}{D_{3}},  \tag{4.31}
\end{equation}%
\begin{equation}
\sum\limits_{n=0}^{\infty }S_{n-1}\left( a_{1}+\left[ -a_{2}\right] \right)
S_{2n+1}\left( e_{1}+\left[ -e_{2}\right] \right) z^{n}=\frac{p\left(
p^{2}+4q\right) z-4p^{2}q^{2}z^{2}}{D_{3}},  \tag{4.32}
\end{equation}%
with 
\begin{equation*}
D_{3}=1-p\left( p^{2}+4q\right) z-2q\left( p^{4}+6p^{2}q+8q^{2}\right)
z^{2}+8pq^{3}\left( p^{2}+4q\right) z^{3}+64q^{6}z^{4}.
\end{equation*}%
And we deduce the following corollaries and theorems.

\begin{corollary}
For $n\in 
\mathbb{N}
$, the new generating function of the product of $\left( p,q\right) $%
-Jacobsthal numbers $\left( J_{p,q,n}J_{p,q,2n}\right) $ is given by:%
\begin{eqnarray}
\dsum\limits_{n=0}^{\infty }J_{p,q,n}J_{p,q,2n}z^{n} &=&\frac{pz+8pq^{3}z^{3}%
}{D_{3}}  \notag \\
&=&\frac{pz+8pq^{3}z^{3}}{1-p\left( p^{2}+4q\right) z-2q\left(
p^{4}+6p^{2}q+8q^{2}\right) z^{2}+8pq^{3}\left( p^{2}+4q\right)
z^{3}+64q^{6}z^{4}},  \TCItag{4.33}
\end{eqnarray}%
with $J_{p,q,n}J_{p,q,2n}=S_{n-1}(a_{1}+[-a_{2}])S_{2n-1}(e_{1}+[-e_{2}]).$
\end{corollary}

\begin{corollary}
Let $n$ be a natural number. Then we have the new generating function of the
product of $\left( p,q\right) $-Jacobsthal numbers $\left(
J_{p,q,n}J_{p,q,2n+1}\right) $:%
\begin{eqnarray}
\dsum\limits_{n=0}^{\infty }J_{p,q,n}J_{p,q,2n+1}z^{n} &=&\frac{\left(
p^{2}+2q\right) z-4pq^{2}z^{2}-16q^{4}z^{3}}{D_{3}}  \notag \\
&=&\frac{\left( p^{2}+2q\right) z-4pq^{2}z^{2}-16q^{4}z^{3}}{1-p\left(
p^{2}+4q\right) z-2q\left( p^{4}+6p^{2}q+8q^{2}\right) z^{2}+8pq^{3}\left(
p^{2}+4q\right) z^{3}+64q^{6}z^{4}},  \TCItag{4.34}
\end{eqnarray}%
with $J_{p,q,n}J_{p,q,2n+1}=S_{n-1}(a_{1}+[-a_{2}])S_{2n}(e_{1}+[-e_{2}]).$
\end{corollary}

\begin{theorem}
For $n\in 
\mathbb{N}
$, the new generating function of the product of $\left( p,q\right) $%
-Jacobsthal Lucas numbers $\left( j_{p,q,n}j_{p,q,2n}\right) $ is given by:%
\begin{equation}
\dsum\limits_{n=0}^{\infty }j_{p,q,n}j_{p,q,2n}z^{n}=\frac{4-3p\left(
p^{2}+4q\right) z-4q\left( p^{4}+6p^{2}q+8q^{2}\right) z^{2}+8pq^{3}\left(
p^{2}+4q\right) z^{3}}{1-p\left( p^{2}+4q\right) z-2q\left(
p^{4}+6p^{2}q+8q^{2}\right) z^{2}+8pq^{3}\left( p^{2}+4q\right)
z^{3}+64q^{6}z^{4}}.  \tag{4.35}
\end{equation}
\end{theorem}

\begin{proof}
We have%
\begin{eqnarray*}
\dsum\limits_{n=0}^{\infty }j_{p,q,n}j_{p,q,2n}z^{n}
&=&\dsum\limits_{n=0}^{\infty }\left( 
\begin{array}{c}
\left( 2S_{n}(a_{1}+[-a_{2}])-pS_{n-1}(a_{1}+[-a_{2}])\right) \\ 
\times \left( 2S_{2n}(e_{1}+[-e_{2}])-pS_{2n-1}(e_{1}+[-e_{2}])\right)%
\end{array}%
\right) z^{n} \\
&=&4\dsum\limits_{n=0}^{\infty
}S_{n}(a_{1}+[-a_{2}])S_{2n}(e_{1}+[-e_{2}])z^{n} \\
&&-2p\dsum\limits_{n=0}^{\infty
}S_{n}(a_{1}+[-a_{2}])S_{2n-1}(e_{1}+[-e_{2}])z^{n} \\
&&-2p\dsum\limits_{n=0}^{\infty
}S_{n-1}(a_{1}+[-a_{2}])S_{2n}(e_{1}+[-e_{2}])z^{n} \\
&&+p^{2}\dsum\limits_{n=0}^{\infty
}S_{n-1}(a_{1}+[-a_{2}])S_{2n-1}(e_{1}+[-e_{2}])z^{n}.
\end{eqnarray*}

Using the relationships $\left( 4.27\right) ,$ $\left( 4.28\right) ,$ $%
\left( 4.30\right) $ and $\left( 4.31\right) $, we obtain%
\begin{eqnarray*}
\dsum\limits_{n=0}^{\infty }j_{p,q,n}j_{p,q,2n}z^{n} &=&\frac{4\left(
1-2pqz-4q^{2}\left( p^{2}+2q\right) z^{2}\right) }{D_{3}}-\frac{2p\left(
p^{2}z+2pq\left( p^{2}+4q\right) z^{2}\right) }{D_{3}} \\
&&-\frac{2p\left( \left( p^{2}+2q\right) z-4pq^{2}z^{2}-16q^{4}z^{3}\right) 
}{D_{3}}+\frac{p^{2}\left( pz+8pq^{3}z^{3}\right) }{D_{3}} \\
&=&\frac{4-3p\left( p^{2}+4q\right) z-4q\left( p^{4}+6p^{2}q+8q^{2}\right)
z^{2}+8pq^{3}\left( p^{2}+4q\right) z^{3}}{D_{3}} \\
&=&\frac{4-3p\left( p^{2}+4q\right) z-4q\left( p^{4}+6p^{2}q+8q^{2}\right)
z^{2}+8pq^{3}\left( p^{2}+4q\right) z^{3}}{1-p\left( p^{2}+4q\right)
z-2q\left( p^{4}+6p^{2}q+8q^{2}\right) z^{2}+8pq^{3}\left( p^{2}+4q\right)
z^{3}+64q^{6}z^{4}}.
\end{eqnarray*}

So, the proof is completed.
\end{proof}

\begin{theorem}
Let $n$\ be a natural number. Then we have the new generating function of
the product of $\left( p,q\right) $-Jacobsthal Lucas numbers $\left(
j_{p,q,n}j_{p,q,2n+1}\right) $:%
\begin{equation}
\dsum\limits_{n=0}^{\infty }j_{p,q,n}j_{p,q,2n+1}z^{n}=\frac{2p-p^{2}\left(
p^{2}+2q\right) z+12pq^{2}\left( p^{2}+4q\right) z^{2}-16p^{2}q^{4}z^{3}}{%
1-p\left( p^{2}+4q\right) z-2q\left( p^{4}+6p^{2}q+8q^{2}\right)
z^{2}+8pq^{3}\left( p^{2}+4q\right) z^{3}+64q^{6}z^{4}}.  \tag{4.36}
\end{equation}
\end{theorem}

\begin{proof}
We have%
\begin{eqnarray*}
\dsum\limits_{n=0}^{\infty }j_{p,q,n}j_{p,q,2n+1}z^{n}
&=&\dsum\limits_{n=0}^{\infty }\left( 
\begin{array}{c}
\left( 2S_{n}(a_{1}+[-a_{2}])-pS_{n-1}(a_{1}+[-a_{2}])\right) \\ 
\times \left( 2S_{2n+1}(e_{1}+[-e_{2}])-pS_{2n}(e_{1}+[-e_{2}])\right)%
\end{array}%
\right) z^{n} \\
&=&4\dsum\limits_{n=0}^{\infty
}S_{n}(a_{1}+[-a_{2}])S_{2n+1}(e_{1}+[-e_{2}])z^{n} \\
&&-2p\dsum\limits_{n=0}^{\infty
}S_{n}(a_{1}+[-a_{2}])S_{2n}(e_{1}+[-e_{2}])z^{n} \\
&&-2p\dsum\limits_{n=0}^{\infty
}S_{n-1}(a_{1}+[-a_{2}])S_{2n+1}(e_{1}+[-e_{2}])z^{n} \\
&&+p^{2}\dsum\limits_{n=0}^{\infty
}S_{n-1}(a_{1}+[-a_{2}])S_{2n}(e_{1}+[-e_{2}])z^{n}.
\end{eqnarray*}

Using the relationships $\left( 4.28\right) ,$ $\left( 4.29\right) ,$ $%
\left( 4.31\right) $ and $\left( 4.32\right) $, we obtain%
\begin{eqnarray*}
\dsum\limits_{n=0}^{\infty }j_{p,q,n}j_{p,q,2n+1}z^{n} &=&\frac{4\left(
p+8pq^{3}z^{2}\right) }{D_{3}}-\frac{2p\left( 1-2pqz-4q^{2}\left(
p^{2}+2q\right) z^{2}\right) }{D_{3}} \\
&&-\frac{2p\left( p\left( p^{2}+4q\right) z-4p^{2}q^{2}z^{2}\right) }{D_{3}}+%
\frac{p^{2}\left( \left( p^{2}+2q\right) z-4pq^{2}z^{2}-16q^{4}z^{3}\right) 
}{D_{3}} \\
&=&\frac{2p-p^{2}\left( p^{2}+2q\right) z+4pq^{2}\left( p^{2}+4q\right)
z^{2}-16p^{2}q^{4}z^{3}}{D_{3}} \\
&=&\frac{2p-p^{2}\left( p^{2}+2q\right) z+12pq^{2}\left( p^{2}+4q\right)
z^{2}-16p^{2}q^{4}z^{3}}{1-p\left( p^{2}+4q\right) z-2q\left(
p^{4}+6p^{2}q+8q^{2}\right) z^{2}+8pq^{3}\left( p^{2}+4q\right)
z^{3}+64q^{6}z^{4}}.
\end{eqnarray*}

So, the proof is completed.
\end{proof}

\begin{corollary}
Setting $p=k$ and $q=1$ in Eqs. (4.33)-(4.36) gives the following new
generating functions:%
\begin{eqnarray*}
\dsum\limits_{n=0}^{\infty }J_{k,n}J_{k,2n}z^{n} &=&\frac{kz+8kz^{3}}{%
1-k\left( k^{2}+4\right) z-2\left( k^{4}+6k^{2}+8\right) z^{2}+8k\left(
k^{2}+4\right) z^{3}+64z^{4}}. \\
\dsum\limits_{n=0}^{\infty }J_{k,n}J_{k,2n+1}z^{n} &=&\frac{\left(
k^{2}+2\right) z-4kz^{2}-16z^{3}}{1-k\left( k^{2}+4\right) z-2\left(
k^{4}+6k^{2}+8\right) z^{2}+8k\left( k^{2}+4\right) z^{3}+64z^{4}}. \\
\dsum\limits_{n=0}^{\infty }j_{k,n}j_{k,2n}z^{n} &=&\frac{4-3k\left(
k^{2}+4\right) z-4\left( k^{4}+6k^{2}+8\right) z^{2}+8k\left( k^{2}+4\right)
z^{3}}{1-k\left( k^{2}+4\right) z-2\left( k^{4}+6k^{2}+8\right)
z^{2}+8k\left( k^{2}+4\right) z^{3}+64z^{4}}. \\
\dsum\limits_{n=0}^{\infty }j_{k,n}j_{k,2n+1}z^{n} &=&\frac{2k-k^{2}\left(
k^{2}+2\right) z+12k\left( k^{2}+4\right) z^{2}-16k^{2}z^{3}}{1-k\left(
k^{2}+4\right) z-2\left( k^{4}+6k^{2}+8\right) z^{2}+8k\left( k^{2}+4\right)
z^{3}+64z^{4}}.
\end{eqnarray*}
\end{corollary}

\begin{itemize}
\item Put $k=1$ in the Corollary 9, we obtain the following table:%
\begin{equation*}
\begin{tabular}{|c|c|}
\hline
$\text{Coefficient of }z^{n}$ & Generating function \\ \hline
$J_{n}J_{2n}$ & $\frac{z+8z^{3}}{1-5z-30z^{2}+40z^{3}+64z^{4}}$ \\ \hline
$J_{n}J_{2n+1}$ & $\frac{3z-4z^{2}-16z^{3}}{1-5z-30z^{2}+40z^{3}+64z^{4}}$
\\ \hline
$j_{n}j_{2n}$ & $\frac{4-15z-60z^{2}+40z^{3}}{1-5z-30z^{2}+40z^{3}+64z^{4}}$
\\ \hline
$j_{n}j_{2n+1}$ & $\frac{2-3z+60z^{2}-16z^{3}}{1-5z-30z^{2}+40z^{3}+64z^{4}}$
\\ \hline
\end{tabular}%
\end{equation*}
\end{itemize}

\begin{center}
\textbf{Table 7.} A new generating functions of the products of some numbers.
\end{center}

\section{\textbf{Conclusion}}

This study proposes to present new class of generating functions for some
special numbers with parameters $p$ and $q$ by using the symmetric
functions. We can summarize the sections as follows:

\begin{itemize}
\item \textbf{In Section 1,} we presented some backgrounds about $\left(
p,q\right) $-numbers and some preliminary facts and results on the\emph{\ }%
symmetric functions.

\item \textbf{In Section 2, }we derived and proved 5 theorems (1-5) by
making use of the symmetrizing operator given by Definition 5.

\item \textbf{In Section 3, }By making use the theorems given in Section 2,
we have derived some new generating functions of odd and even terms of $%
\left( p,q\right) $-numbers.

\item \textbf{In Section 4, }By making use the symmetric functions, we
investigated the new generating functions of the products of $\left(
p,q\right) $-numbers with odd and even $\left( p,q\right) $-numbers.
\end{itemize}

\end{document}